\documentclass[12pt]{article}

\usepackage{amsmath,amssymb,amsthm,amsxtra,mathrsfs}
\usepackage[shortlabels]{enumitem}
\setlist{topsep=0em,partopsep=0em,parsep=0em,itemsep=0em}
\usepackage{hyperref}
\usepackage{tikz}
\usepackage{tikz-cd}
\usepackage[margin=1in]{geometry}
\usepackage{stmaryrd}

\numberwithin{equation}{section}

\newtheorem{theorem}{Theorem}[section]
\newtheorem{lemma}[theorem]{Lemma}
\newtheorem{proposition}[theorem]{Proposition}
\newtheorem{corollary}[theorem]{Corollary}

\theoremstyle{definition}
\newtheorem{definition}[theorem]{Definition}
\newtheorem{remark}[theorem]{Remark}
\newtheorem{question}[theorem]{Question}
\newtheorem*{convention}{Conventions}

\newcommand*{\C}{\mathbb{C}}
\newcommand*{\Q}{\mathbb{Q}}
\newcommand*{\Z}{\mathbb{Z}}
\newcommand*{\R}{\mathbb{R}}
\newcommand*{\N}{\mathbb{N}}
\newcommand*{\cp}{\mathbb{P}}
\newcommand*{\sei}{\mathbb{S}}
\newcommand*{\seiht}{\widehat{\mathbb{S}}}

\newcommand*{\td}{\widetilde}
\newcommand*{\wht}{\widehat}

\newcommand*{\id}{\mathrm{id}}
\newcommand*{\ev}{\mathrm{ev}}
\newcommand*{\pt}{\mathrm{pt}}
\newcommand*{\End}{\operatorname{End}}
\newcommand*{\Hom}{\operatorname{Hom}}
\newcommand*{\rank}{\operatorname{rank}}
\newcommand*{\pr}{\mathrm{pr}}
\newcommand*{\U}{\mathrm{U}}
\newcommand*{\ft}{\mathrm{FT}}
\newcommand*{\lt}{\mathrm{LT}}
\newcommand*{\qdm}{\mathrm{QDM}}
\newcommand*{\res}{\mathrm{Res}}
\newcommand*{\spf}{\operatorname{Spf}}
\newcommand*{\eff}{C^\textnormal{eff}}
\newcommand*{\ch}{\mathrm{ch}}
\newcommand*{\Td}{\mathrm{Td}}

\newcommand*{\gsp}{\mathcal{H}}
\newcommand*{\gcn}{\mathcal{L}}
\newcommand*{\lb}{\mathcal{L}}
\newcommand*{\sft}{\mathcal{S}}
\newcommand*{\sftht}{\widehat{\mathcal{S}}}
\newcommand*{\pot}{\mathcal{F}}
\newcommand*{\bigj}{\mathcal{J}}
\newcommand*{\sq}{\mathcal{Q}}
\newcommand*{\mdl}{\mathcal{M}}
\newcommand*{\nm}{N}
\newcommand*{\vnm}{\mathcal{N}^\textnormal{vir}}
\newcommand*{\thck}{\mathcal{T}}
\newcommand*{\obs}{\mathcal{E}}
\newcommand*{\fms}{\mathcal{F}}
\newcommand*{\uvf}{\mathcal{C}}
\newcommand*{\zlc}{\mathcal{Z}}
\newcommand*{\shf}{\mathcal{O}}
\newcommand*{\ofd}{\mathcal{X}}
\newcommand*{\nds}{\mathcal{S}}

\newcommand*{\nov}[1]{\C\llbracket #1\rrbracket}
\newcommand*{\fps}[1]{\llbracket #1\rrbracket}
\newcommand*{\extnov}[2]{(\!({#1})\!)\llbracket #2\rrbracket}
\newcommand*{\laur}[1]{(\!({#1})\!)}
\newcommand*{\vfc}[1]{[#1]^\textnormal{vir}}

\newcommand*{\btau}{\boldsymbol{\tau}}
\newcommand*{\bth}{\boldsymbol{\theta}}
\newcommand*{\bt}{\mathbf{t}}
\newcommand*{\cc}{\mathbf{c}}
\newcommand*{\ff}{\mathbf{f}}
\newcommand*{\cft}{\mathscr{F}}
\newcommand*{\dft}{\mathsf{F}}
\newcommand*{\fds}{\mathsf{H}}
\newcommand*{\idl}{\mathfrak{m}}
\newcommand*{\q}{\mathfrak{q}}

\title{The quartic threefold is symplectically irrational}
\author{Jiaji Cai}
\date{}

\begin{document}

\maketitle

\begin{abstract}
	We prove that smooth quartic threefolds are symplectically irrational: they cannot be related to projective space by a sequence of symplectic blow-ups, blow-downs, and deformations. This recovers the classical irrationality theorem of Iskovskikh--Manin. The obstruction is constructed from big quantum cohomology, using the multiplicities of eigenvalues of quantum multiplication by the Euler vector field. To prove its invariance, we establish a decomposition theorem for quantum cohomology under symplectic blow-ups, following the work of Iritani.
\end{abstract}


\section{Introduction}

A smooth complex projective variety is said to be rational if it is isomorphic to projective space outside Zariski closed subsets. The weak factorization theorem \cite{AKMW02} implies that this is equivalent to the existence of a sequence of blow-ups and blow-downs relating the variety to $\cp^n$. This alternative characterization can be generalized to symplectic manifolds.
\begin{definition}\label{def:rat}
	Two symplectic manifolds are \emph{symplectic birational equivalent} if they are related by a finite sequence of symplectic blow-ups, blow-downs, and symplectic deformations, in any order. A symplectic manifold is \emph{symplectically rational} if it is symplectic birational equivalent to complex projective space; otherwise it is said to be \emph{symplectically irrational}.
\end{definition}

If a smooth complex projective variety is algebraically rational, then by the weak factorization theorem it is also symplectically rational.

A symplectic blow-up is well-defined up to deformation, so it is natural to allow deformations in the definition of a symplectic birational equivalence, thereby giving the notion a more topological flavor.
Moreover, Gromov's h-principle for isosymplectic embeddings \cite{Gro86} implies that in codimension  at least $4$ there are far more symplectic submanifolds than algebraic subvarieties. Consequently, symplectic blow-ups can be considerably wilder than their algebraic counterparts; see, for example, \cite{McD84}.

\begin{question}\label{qst}
	Which smooth complex hypersurfaces in $\cp^{n+1}$ are symplectically rational?
\end{question}
Smooth complex hypersurfaces of a fixed degree are all symplectomorphic. Those of degree at most 2 are algebraically rational, and hence symplectically rational. On the other hand, those of degree at least $n+2$ are symplectically irrational: they are not uniruled, and uniruledness is a symplectic birational invariant by the results of Koll\'ar \cite{Kol98}, Ruan \cite{Rua99}, and Hu--Li--Ruan \cite{HLR08}. Thus, for $n\le3$, the only remaining cases whose symplectic rationality is unknown are the cubic and quartic hypersurfaces in $\cp^4$.
In \cite{Smi23}, Smith asked whether the cubic threefold is symplectically irrational. While we do not answer this question here, we prove the corresponding statement for the quartic threefold:

\begin{theorem}\label{thm:main}
	The quartic threefold is symplectically irrational.
\end{theorem}
In particular, this implies that smooth quartic threefolds are algebraically irrational, thereby recovering the classical result of Iskovskikh--Manin \cite{IM71}. In fact, they prove that every birational automorphism of a smooth quartic threefold is biregular, from which irrationality follows. Their argument is inherently algebro-geometric and does not appear to have a direct analogue in the symplectic setting.

We briefly recall the argument of \cite{HLR08}. By \cite{Kol98,Rua99}, uniruledness can be characterized in terms of Gromov--Witten invariants. The proof that uniruledness is a symplectic birational invariant proceeds by regarding a symplectic blow-up as one side of a symplectic cut and applying the degeneration formula. To our knowledge, non-uniruledness is the only previously known obstruction to symplectic rationality available for projective hypersurfaces. Since the quartic threefold is uniruled, new obstructions are needed.

Recall that big quantum cohomology carries a distinguished element, the Euler vector field. We study the operator given by quantum multiplication by this element. For the quartic threefold, we show that this operator has an eigenvalue of multiplicity $3$. On the other hand, a decomposition theorem for big quantum cohomology under symplectic blow-ups, stated below as Theorem \ref{thm:blowup}, implies that for any symplectically rational 6-manifold, all such eigenvalues have multiplicity at most $2$. Indeed, the relevant eigenvalues are simple for $\cp^3$, and the only possible blow-up centers in dimension 6 are points and symplectic surfaces, whose contributions have multiplicity at most $2$. This contradiction proves Theorem \ref{thm:main}. Throughout this argument, we restrict attention to even cohomology and ignore odd classes.

We now explain the relation of this work to that of Katzarkov--Kontsevich--Pantev--Yu \cite{KKPY25}. Their key idea is to combine Gromov--Witten theory with Hodge theory in order to construct a birational invariant which obstructs rationality for a dense subset of cubic fourfolds in the deformation family. In the symplectic setting, Hodge theory is no longer available. Moreover, symplectic rationality is a property of the deformation class. Smooth cubic fourfolds are symplectically rational, since they form a single deformation family containing rational examples. We therefore construct a symplectic birational invariant using only Gromov--Witten theory. This forces us to consider eigenvalues arising from big quantum cohomology, whereas in \cite{KKPY25}, because of the Hodge-theoretic input, it is enough to use information from small quantum cohomology.

Let us elaborate on this point. The fact that quantum multiplication by the Euler vector field for the quartic threefold has an eigenvalue of multiplicity 3 is quite remarkable. It is well known that small quantum multiplication by the first Chern class has such an eigenvalue. A priori, however, one might expect this eigenvalue to split under the big quantum cohomology deformation into several eigenvalues of smaller multiplicities. We prove that this does not occur for the quartic threefold (Theorem \ref{thm:eigen}). The key point is that the eigenvalues are determined by their initial values on small quantum cohomology: using the Cayley--Hamilton theorem and commutator identities for powers of the Euler vector field, we derive a system of ODEs governing the eigenvalues. This argument applies whenever the even cohomology ring is generated by the first Chern class.


The decomposition theorem was proved by Iritani for K\"ahler blow-ups of smooth projective varieties \cite{Iri25}. Once the virtual localization formula \cite{Kon95} and the quantum Riemann--Roch theorem \cite{CG07} are established in the symplectic setting, Iritani's proof generalizes directly. For the necessary symplectic virtual techniques, we use global Kuranishi charts, introduced by Abouzaid--McLean--Smith \cite{AMS21}.

\medskip

\noindent {\bf Plan of the paper.} In Section \ref{sec:blowup}, we recall the construction of symplectic blow-ups. In Section \ref{sec:gw}, we state the decomposition theorem and discuss the various aspects of Gromov--Witten theory used in its proof. In Section \ref{sec:decomp}, we give the proof of the decomposition theorm. In Section \ref{sec:irr}, we prove our result on symplectic irrationality. In Appendix \ref{sec:app}, we discuss symplectic virtual technqiues and prove the quantum Riemann--Roch theorem.

\medskip

\noindent {\bf Acknowledgements.} I am grateful to Mark McLean for his constant guidance, for reading through this paper and making important suggestions. I am indebted to Ivan Smith for pointing out \cite[Remark 3.14]{KKPY25}, which led to the proof of Theorem \ref{thm:eigen}. I also thank Hiroshi Iritani for kind correspondences in regard to my earlier attempts at proving this theorem. I thank John Pardon for suggestions on the exposition of the paper. I thank Mohammed Abouzaid and Shuhao Li for conversations which eventually led me to consider this problem in the beginning. The author was partially supported by NSF award DMS-2203308 and Simons Foundation
International, LTD.

\begin{convention}
	All manifolds are assumed to be smooth, connected, compact, and without boundary, unless there is clear indication of the contrary. We usually refer to real dimensions, although we adopt standard terminology like ``threefold'', ``line bundle'', etc.
	
	{\bf Important convention}: We will only work with cohomology groups in even degrees, so by abuse of notation we let $H^*(X)$ be the direct sum (or union, depending on context) of even degree cohomology groups of $X$. The coefficient ring will by default be $\C$, unless explicitly stated otherwise.
	
	We always use first Chern class of the tangent bundle $c_1(X)=c_1(TX)$. 
	
	We let $\N=\Z_{\ge0}$ denote the set of non-negative integers.
\end{convention}

\section{Symplectic blow-ups}\label{sec:blowup}

We recall the definition of a symplectic blow-up along a symplectic submanifold following \cite{GS89,MS98}. 

Let $(X,\omega)$ be a symplectic manifold, $Z\subset X$ a symplectic submanifold of codimension $2r$. The normal bundle of $Z$ has a compatible complex structure, and thus can be identified with $P\times_{\U(r)}\C^r\to Z$, where $P\to Z$ is a principal $\U(r)$-bundle. The group $\U(r)$ acts in a Hamiltonian way on $(\C^r,\omega_0)$, and on the blow-up $(\td\C^r_\epsilon,\td\omega_\epsilon)$ of the origin. Let $\mu,\td\mu$ be the respective moment maps.

We wish to construct a closed 2-form $\td\alpha_\epsilon$ on $P\times_{\U(r)}\td\C^r_\epsilon$ whose restriction to each fiber is $\td\omega_\epsilon$. To do this, we choose a connection $A\in\Omega^1(P,\mathfrak{u}(r))$. Consider the 2-form on $P\times\td\C^r_\epsilon$ given by $\td\omega_\epsilon-d\langle\td\mu,A\rangle$ (where pull-backs are omitted in the notation). It is obviously closed and $\U(r)$-invariant; more importantly, its contraction with any vector field $X_\xi$ generated by $\xi\in\mathfrak{u}(r)$ is zero. This means that it descends to a closed 2-from on the quotient $P\times_{\U(r)}\td\C^r_\epsilon$, which we denote by $\td\alpha_\epsilon$. 

An analogous construction defines a closed 2-form $\alpha$ on $P\times_{\U(r)}\C^r$ which restricts to $\omega_0$ on each fiber. Then $\alpha+\omega_Z$ is a symplectic form on a neighborhood of the zero section, as it is non-degenerate on each fiber and equals $\omega_Z$ on the zero section.
By the symplectic neighborhood theorem, $Z\subset M$ has a tubular neighborhood $N_{\epsilon_0}(Z)$ which is symplectomorphic to the disc bundle $P\times_{\U(r)}B(\epsilon_0)$ with symplectic form $\alpha+\omega_Z$. If $\epsilon<\epsilon_0$, then the $\epsilon_0$-neighborhood $\td B(\epsilon_0)$ of the exceptional $\cp^{r-1}$ in $\td\C^r_\epsilon$ has boundary symplectomorphic to $(\partial B(\epsilon_0),\omega_0)$. This implies that the boundary of $N_{\epsilon_0}(Z)$ is symplectomorphic to the boundary of $P\times_{\U(r)}\td B(\epsilon_0)$, where the latter is equipped with symplectic form $\td\alpha_\epsilon+\omega_Z$.
\begin{definition}
	The \emph{symplectic blow-up} $\td X$ is obtained from $X$ by cutting out $N_{\epsilon_0}(Z)$ and gluing in $P\times_{\U(r)}\td B(\epsilon_0)$ equipped with the symplectic form $\td\alpha_\epsilon+\omega_Z$.
\end{definition}
Once $\epsilon_0,\epsilon$ are sufficiently small, the resulting symplectic structure is independent of the various choices up to isotopy.

Next, we recall the additive structure of cohomology (with coefficients in $\C$) of $\td X$. Let $\varphi:\td X\to X$ be the blow-down map, $D=\varphi^{-1}(Z)$ the exceptional divisor, $\iota:Z\to X$ and $\jmath:D\to\td X$ the inclusion maps, $\pi=\varphi|_D:D\to Z$ the projection map, and $p\in H^2(D)$ the unique cohomology class which restricts to the generator of cohomology on each fiber, and restricts to 0 on $D-Z$.

\begin{proposition}[{\cite[Proposition 2.4]{McD84}}]\label{prop:addiso}
	There is an additive isomorphism \begin{equation*}
	H^*(X)\oplus\bigoplus_{k=0}^{r-2}H^*(Z)\xrightarrow[\textrm{as modules}]{\simeq} H^*(\td X)
\end{equation*} where $\alpha\in H^*(X)$ is sent to $\varphi^*\alpha$, and $\beta$ from the $k$th copy of $H^*(Z)$ is sent to $\jmath_*(p^k\pi^*\beta)$. Here $\jmath_*$ is the Gysin push-forward $\mathrm{PD}_{\td X}^{-1}\circ\jmath_*\circ\mathrm{PD}_D.$
\end{proposition}

Clearly, the isomorphism of \cite{McD84} restricts to an isomorphism on even cohomology.

%
%

\section{Gromov--Witten theory}\label{sec:gw}
In this section we give the precise statement of the decomposition theorem, and review the various aspects of Gromov--Witten theory that will be used in its proof. In particular, we will recall the virtual localization formula \cite{Kon95} and the quantum Riemann--Roch theorem \cite{CG07}. We will also define Iritani's extended shift operators \cite{Iri25} in the symplectic setting. Symplectic virtual techniques are discussed in Appendix \ref{sec:app}.

\subsection{Gromov--Witten invariants}

Let $(X,\omega)$ be a symplectic manifold. 

\begin{definition}
	Let $X_{0,n,d}$ denote the moduli space of stable genus zero pseudo-holomorphic maps to $X$ with $n$ marked points, representing the homology class $d\in H_2(X,\Z)$, with respect to some tame almost complex structure. Let $[X_{0,n,d}]^\textrm{vir}$ denote its virtual fundamental class. 
	
	Let $\ev_i:X_{0,n,d}\to X$ denote the evaluation map at the $i$th marked point. Let $\mathcal L_i$ denote the $i$th cotangent line bundle on $X_{0,n,d}$. Its fiber at $[C,x_1,\dots,x_n,u]\in X_{0,n,d}$ is the cotangent space $T_{x_i}^*C$. Let $\psi_i:=c_1(\mathcal L_i).$
	
	For $\alpha_1,\dots,\alpha_n\in H^*(X,\Q)$ and $ k_1,\dots,k_n\in\N$, the \emph{(genus zero) descendant Gromov--Witten invariant} is \[\langle\alpha_1\psi^{k_1},\dots,\alpha_n\psi^{k_n}\rangle_{0,n,d}^X:=\int_{[X_{0,n,d}]^\textrm{vir}}\prod_{i=1}^n\ev_i^*(\alpha_i)\psi_i^{k_i}\in\Q.\] When $k_1=\dots=k_n=0$, we call these the \emph{primary Gromov--Witten invariants} \[\langle\alpha_1,\dots,\alpha_n\rangle_{0,n,d}^X\in\Q.\] 
	
	For the details of this definition, see Appendix \ref{sec:app}.
\end{definition}

We now fix our conventions for completions. Let $M=\bigoplus_{n\in\Z}M_n$ be a $\Z$-graded module. Let $M^k=\bigoplus_{n\in\Z}M^k_n\subseteq M$ be a descending chain of graded submodules. Then the \emph{(graded) completion} of $M$ is defined to be \[\widehat M:=\bigoplus_{n\in\Z}\widehat M_n,\quad \widehat M_n:=\varprojlim_kM_n/M_n^k.\]
That is to say, if we write an element of $\widehat M$ as an infinite sum, then it is not allowed to have elements of infinitely many different degrees. In this paper, completions are always graded completions.

For a symplectic manifold $(X,\omega)$, we let $\eff(X,\omega)$ be the monoid generated by classes $d\in H_2(X,\Z)$ with the property that some Gromov--Witten invariant $\langle\alpha_1,\dots,\alpha_n\rangle_{0,n,d}^X$ is nonzero. Let $\C[Q]$ denote the group ring of $\eff(X,\omega)$ over $\C$. The element corresponding to $d\in H_2(X,\Z)$ is denoted by $Q^d$, and defined to have grading $\deg Q^d:=2c_1(X)\cdot d.$ We define $\nov{Q}$ to be the graded completion of $\C[Q]$ with respect to the descending chain of submodules $\left( Q^ d:[\omega]\cdot d\ge k\right)$ for $k\in\N$. The ring $\nov Q$ is called the \emph{Novikov ring}, and $Q$ is called the \emph{Novikov variable.} A standard argument involving Gromov compactness shows that the Novikov ring only depends on the deformation class of $\omega$.

Let $\phi_0=1,\phi_1,\dots,\phi_s$ be homogeneous elements forming a $\C$-basis of $H^*(X)$. An element in $H^*(X)$ is of the form $\tau=\sum_i\tau^i\phi_i.$ We set $\deg\tau^i:=2-\deg\phi_i$, making the polynomial ring $\nov{Q}[\tau^0,\dots,\tau^s]$ a graded ring. Then we define $\nov{Q,\tau}:=\nov{Q}\fps{\tau^0,\dots,\tau^s}$ to be the graded completion with respect to the descending chain of submodules $\left((\tau^0)^k,\dots,(\tau^s)^k\right)$ for $k\in\N$. The variables $\tau^i$ are called \emph{bulk variables.}

\begin{definition}
	Let $(\cdot,\cdot)_X:H^*(X)\otimes H^*(X)\to \C$ denote the Poincar\'e pairing, which we extend linearly over $\nov{Q,\tau}$. The \emph{(big) quantum product} $\star$ is a $\nov{Q,\tau}$-bilinear commutative associative product on $H^*(X)\fps{Q,\tau}$ given by \[(\phi_i\star\phi_j,\phi_k)_X=\sum_{\substack{ d\in H_2(X,\Z)\\ n\in\N}}\langle\phi_i,\phi_j,\phi_k,\tau,\dots,\tau\rangle^X_{0,n+3, d}\frac{Q^{d}}{n!}.\] For the right hand side, one should expand the correlators using $\tau=\sum_i\tau^i\phi_i$ to get an element in $\nov{Q,\tau}$.
\end{definition}

The space $H^*(X)\fps{Q,\tau}$ is equipped with a \emph{Frobenius structure}, which can be captured in terms of the \emph{quantum connection}. To define this, we introduce an auxiliary variable $z$ of degree 2, called the \emph{loop variable}. It can be viewed as the equivariant parameter of the $S^1$-action rotating the domain of maps $\cp^1\to X$.

\begin{definition}
	The \emph{quantum connection} $\nabla$ is given by the operators $\nabla_{\tau^i},\nabla_{z\partial_z},\nabla_{\xi Q\partial_Q}:H^*(X)[z]\fps{Q,\tau}\to z^{-1}H^*(X)[z]\fps{Q,\tau}$ (for $\xi\in H^2(X)$) \begin{align*}
		&\nabla_{\tau^i}=\partial_{\tau^i}+z^{-1}\phi_i\star,\\
		&\nabla_{z\partial_z}=z\partial_z-{z^{-1}E_X\star}+\mu_X,\\
		&\nabla_{\xi Q\partial_Q}=\xi Q\partial_Q+z^{-1}{\xi\star},
	\end{align*} where $E_X\in H^*(X)\fps{Q,\tau}$ is the \emph{Euler vector field} \[E_X=c_1(X)+\sum_i\left(1-\frac{\deg\phi_i}{2}\right)\tau^i\phi_i,\] $\mu_X\in\End(H^*(X))$ is the \emph{grading operator} \[\mu_X(\phi_i)=\left(\frac{\deg\phi_i}{2}-\frac{\dim X}{4}\right)\phi_i,\] and $\xi Q\partial_Q$ is the derivation on $\nov Q$ given by \[(\xi Q\partial_Q)Q^d=(\xi\cdot d)Q^d.\]
	
	Define the $z$-sesquilinear pairing $P_X$ by \[P_X(f,g):=(f(-z),g(z))_X\in\C[z]\fps{Q,\tau}.\] 
	
	The module $H^*(X)[z]\fps{Q,\tau}$ equipped with the {quantum connection} $\nabla$ and the pairing $P_X$ is called the \emph{quantum D-module} of $X$, and denoted by $\qdm(X)$.
\end{definition} 

\begin{remark}
	One may view the quantum connection as a meromorphic connection on the trivial $H^*(X)$-bundle over $\spf\nov{z,Q,\tau}$. The latter can be thought of as the infinitesimal neighborhood of 0 in the affine space with coordinates $z,Q,\tau^0,\dots,\tau^s$, on which the ring of functions is $\nov{z,Q,\tau}$. The $\tau$ directions parametrize bulk deformations, and the $Q$ direction parametrizes deformations of the symplectic form. The variable $z$ is closely related to Givental's formalism.
\end{remark}

These satisfy the following two properties: \begin{enumerate}[(1)]
	\item The quantum connection $\nabla$ is flat. This follows from the associativity, commutativity, and potentiality of the quantum product.
	\item The pairing $P_X$ is compatible with connection, i.e. $dP_X(f,g)=P_X(\nabla f,g)+P_X(f,\nabla g).$ This follows from the symmetry of Gromov--Witten invariants.
\end{enumerate}

The quantum connection has a \emph{fundamental solution} $M_X\in\End(H^*(X))[z^{-1}]\fps{Q,\tau}$ which trivializes the connection in the following sense: \begin{align*}
	&M_X\circ\nabla_{\tau^i}=\partial_{\tau^i}\circ M_X\\
	&M_X\circ\nabla_{z\partial_z}=\left(z\partial_z-z^{-1}{c_1(X)\smile}+\mu_X\right)\circ M_X\\
	&M_X\circ\nabla_{\xi Q\partial_Q}=\left(\xi Q\partial_Q+z^{-1}{\xi\smile}+\mu_X\right)\circ M_X
\end{align*} where $\smile$ denotes the classical cup product. The fundamental solution $M_X$ can be explicitly expressed in terms of descendant Gromov--Witten invariants (reason for introducing descendant invariants in the first place) \[(M_X\phi_i,\phi_j)_X=(\phi_i,\phi_j)_X+\sum_{\substack{d\in H_2(X,\Z),n\in\Z\\(n,d)\ne(0,0)}}\left\langle\phi_i,\tau,\dots,\tau,\frac{\phi_j}{z-\psi}\right\rangle_{0,n+2,d}^X\frac{Q^d}{n!}\] where $\frac{\phi_j}{z-\psi}$ should be expanded as $\sum_{a=0}^\infty\phi_j\psi^az^{-a-1}$. It also preserves the pairing \[P_X(M_Xf,M_Xg)=P_X(f,g).\] The \emph{J-function} of $X$ is defined to be $M_X\phi_0$, and can be explicitly written as \[J_X=1+\frac{\tau}{z}+\sum_{j=0}^s\sum_{\substack{d\in H_2(X,\Z),n\in\Z\\(n,d)\ne(0,0),(1,0)}}\phi^j\left\langle\tau,\dots,\tau,\frac{\phi_j}{z(z-\psi)}\right\rangle_{0,n+2,d}^X\frac{Q^d}{n!}\] where $\{\phi^i\}$ is the basis dual to $\{\phi_i\}$, i.e., $(\phi_i,\phi^j)_X=\delta_i^j$.

\subsection{Decomposition theorem}

We can now state the decomposition theorem for quantum cohomology under symplectic blow-ups. Let $\td X$ be the symplectic blow-up of $X$ along a symplectic submanifold $Z$. Roughly, the theorem says that there is a ring isomorphism between $H^*(\td X)$ and a direct sum of $H^*(X)$ and copies of $H^*(Z)$ with respect to the big quantum products, after extending to a bigger coefficient ring.

Let $\varphi:\td X\to X$ be the blow-down map, $D=\varphi^{-1}(Z)$ the exceptional divisor, $\iota:Z\to X$ the inclusion map. 
Recall that $\td X$ is equipped with a symplectic form $\td\omega_\epsilon$ for sufficiently small $\epsilon$. This is unique up to deformation, so we can unambiguously talk about its Novikov ring $\nov{\td Q}$.
Suppose $Z$ is of codimension $2r$. We introduce a new variable $\q$ of degree $2(r-1)$, which represents the class of a line in $D$ mapping to a point in $Z$. We introduce a new ring $\C\laur{\q^{-\frac{1}{r-1}}}\fps{Q}$, where $Q$ is the Novikov variable of $X$. The Novikov rings of $X,\td X,Z$ can be extended to this new ring in the following way: \begin{equation}\label{eq:novext}
	\begin{split}
	\nov{Q} \hookrightarrow \C\laur{\q^{-\frac{1}{r-1}}}\fps{Q}&, \quad \textrm{the obvious inclusion},\\
	\nov{\td Q} \hookrightarrow \C\laur{\q^{-\frac{1}{r-1}}}\fps{Q}&, \quad \td Q^{\td d}\mapsto Q^{\varphi_*\td d}\q^{-[D]\cdot\td d},\\
	\nov{Q_Z} \to \C\laur{\q^{-\frac{1}{r-1}}}\fps{Q}&, \quad Q_Z^{d}\mapsto Q^{\iota_* d}\q^{-\frac{\rho_Z\cdot d}{r-1}},
	\end{split}
\end{equation} where $\rho_Z$ is the first Chern class of the normal bundle $\nm_{Z/X}$ of $Z$ in $X$. Thus we extend the quantum products of $X,\td X,Z$ to be defined over $\C\laur{\q^{-\frac{1}{r-1}}}\fps{Q}$. Let $\qdm(\cdot)^\textnormal{la}$ denote the base change of $\qdm(\cdot)$ via \eqref{eq:novext}; in the case when $r$ is odd, we also adjoin $\q^{\pm1/2(r-1)}$.

\begin{theorem}\label{thm:blowup}
	Let us temporarily denote $R=\C\laur{\q^{-\frac{1}{r-1}}}\fps{Q,\td\tau}$, where $\td\tau$ are the bulk variables of $\td X$. There is an isomorphism of F-manifolds with Euler vector fields\begin{align*}
		H^*(\td X)\extnov{\q^{-\frac{1}{r-1}}}{Q} &\xrightarrow[\textrm{as F-manifolds}]{\simeq}H^*(X)\extnov{\q^{-\frac{1}{r-1}}}{Q}\oplus H^*(Z)\extnov{\q^{-\frac{1}{r-1}}}{Q}^{\oplus(r-1)},\\
		\td\tau &\mapsto \left(\tau(\td\tau),\varsigma_0(\td\tau),\dots,\varsigma_{r-2}(\td\tau)\right).
	\end{align*} Equivalently, its differential is an isomorphism of algebras \begin{equation*}
		\left(H^*(\td X,R),\star_{\td\tau}\right)\xrightarrow[\textrm{as algebras}]{\simeq} \left(H^*(X,R),\star_{\tau(\td\tau)}\right)\oplus\bigoplus_{j=0}^{r-2} \left(H^*(Z,R),\star_{\varsigma_j(\td\tau)}\right)
	\end{equation*} which sends $E_{\td X}$ to $(E_X,E_Z,\dots,E_Z).$
	
	Moreover, the isomorphism of F-manifolds is covered by an isomorphism of quantum D-modules \[\Psi:\qdm(\td X)^\textnormal{la}\xrightarrow{\simeq}\tau^*\qdm(X)^\textnormal{la}\oplus\bigoplus_{j=0}^{r-2}\varsigma_j^*\qdm(Z)^\textnormal{la}.\] In particular, $\Psi$ intertwines $P_{\td X}$ with $P_X\oplus P_Z^{\oplus(r-1)}.$
\end{theorem}

Concretely, the change of variables can be expressed as a collection of formal functions \[\tau^i=\tau^i(\td\tau^0,\dots,\td\tau^s),\ \varsigma_k^j=\varsigma_k^j(\td\tau^0,\dots,\td\tau^s)\in \C\laur{\q^{-\frac{1}{r-1}}}\fps{Q,\td\tau},\] and the quantum products $\star_{\tau(\td\tau)},\star_{\varsigma_k(\td\tau)}$ are obtained by substituting in the above functions.

\begin{remark}
	For the definition of an F-manifold with Euler vector field, see \cite{Man99}. For the proof of Theorem \ref{thm:main}, we shall only need the isomorphism of algebras; the isomorphism of quantum D-modules is a byproduct of the proof and may be useful for other applications. Note that the isomorphism of algebras need not respect the pairings.
\end{remark}

\subsection{Equivariant Gromov--Witten invariants}

Suppose $(X,\omega)$ has a Hamiltonian action by a torus $T=(S^1)^{\rank T}$. We choose the almost complex structrue to be $T$-invariant, so the moduli space $X_{0,n,d}$ has an induced $T$-action. 

\begin{definition}
	Let $\lambda_1,\dots,\lambda_{\rank T}$ be the equivariant parameters, i.e., $H_T^*(\pt,\Q)=\Q[\lambda]:=\Q[\lambda_1,\dots,\lambda_{\rank T}]$. For $\alpha_1,\dots,\alpha_n\in H_T^*(X,\Q)$ and $ k_1,\dots,k_n\in\N$, the \emph{equivariant descendant Gromov--Witten invariant} is \[\langle\alpha_1\psi^{k_1},\dots,\alpha_n\psi^{k_n}\rangle_{0,n,d}^{X,T}:=\int_{[X_{0,n,d}]^\textrm{vir}}\prod_{i=1}^n\ev_i^*(\alpha_i)\psi_i^{k_i}\in\Q[\lambda].\] 
	
	For the details of this definition, see Appendix \ref{sec:app}.
\end{definition}

We state the \emph{virtual localization formula}, which will be used several times in this paper.
\begin{theorem}\label{thm:loc}
	Let $\{X_{0,n,d}(k)\}_{k}$ be the $T$-fixed components of $X_{0,n,d}$. Then \[[X_{0,n,d}]^\textnormal{vir}=\sum_k\frac{[X_{0,n,d}(k)]^\textnormal{vir}}{e_T(\vnm
	(k))}\] where $\vnm(k)$ is the virtual normal bundle of $X_{0,n,d}(k)$ in $X_{0,n,d}$.
\end{theorem}
In the symplectic setting, this is proven in \cite{Hir25}, using global Kuranishi charts as well.

Let $\phi_0=1,\phi_1,\dots,\phi_s$ be homogeneous elements forming a basis of $H_T^*(X)$ over $H_T^*(\pt)=\C[\lambda]$. An element of $H_T^*(X)$ is of the form $\tau=\sum_i\tau^i\phi_i$ where $\tau^i$ are $\C[\lambda]$-valued variables.

\begin{definition}
	Let $(\cdot,\cdot)_X^T:H_T^*(X)\otimes_{\C[\lambda]}H_T^*(X)\to\C[\lambda]$ denote the equivariant Poincar\'e pairing, which we extend linearly over $\C[\lambda]\fps{Q,\tau}.$ The \emph{equivariant quantum product} $\star$ is the $\C[\lambda]\fps{Q,\tau}$-bilinear commutative associative product on $H_T^*(X)\fps{Q,\tau}$ given by \[(\phi_i\star\phi_j,\phi_k)_X^T=\sum_{\substack{ d\in H_2(X,\Z)\\ n\in\N}}\langle\phi_i,\phi_j,\phi_k,\tau,\dots,\tau\rangle^{X,T}_{0,n+3, d}\frac{Q^{d}}{n!}.\]
\end{definition}

To define the \emph{equivariant quantum connection}, we further introduce $\C$-valued variables $\tau^{i,k}$ for $0\le i\le s,\ k\in\N^{\rank T}$, dual to the $\C$-basis $\{\phi_i\lambda^k\}$ of $H_T^*(X)$. Thus we have $\tau^i=\sum_k\tau^{i,k}\lambda^k.$ We set $\deg\lambda=2$ and $\deg\tau^{i,k}=2-\deg\phi_i-2|k|$ where $|k|=\sum_{a=1}^{\rank T}k_a$. We use the boldface letter $\btau$ to denote the infinite set of variables $\{\tau^{i,k}\}$ (recall that $\tau$ was shorthand for $\{\tau^i\}$). Let $\C[\btau]$ denote the ring of polynomials in the variables $\tau^{i,k}$. We define $\nov\btau$ to be the graded completion of $\C[\btau]$ with respect to the descending chain of submodules, where the $m$th one is generated by $(\tau^{i,k})^m$ for $0\le i\le s,\ 0\le|k|\le m$ and $\tau^{i,k}$ for $0\le i\le s,\ |k|> m$.

We also introduce the loop variable $z$ as before.

\begin{definition}
	The \emph{equivariant quantum connection} $\nabla$ is given by operators $\nabla_{\tau^{i,k}},\nabla_{z\partial_z},\nabla_{\xi Q\partial_Q}:H_T^*(X)[z]\fps{Q,\btau}\to z^{-1}H_T^*(X)[z]\fps{Q,\btau}$ (for $\xi\in H_T^2(X)$) \begin{align*}
		&\nabla_{\tau^{i,k}}=\partial_{\tau^{i,k}}+z^{-1}\phi_i\lambda^k\star,\\
		&\nabla_{z\partial_z}=z\partial_z-{z^{-1}E_X\star}+\mu_X,\\
		&\nabla_{\xi Q\partial_Q}=\xi Q\partial_Q+z^{-1}{\xi\star},
	\end{align*} where $E_X\in H_T^*(X)\fps{Q,\btau}$ is the \emph{Euler vector field} \[E_X=c_1^T(X)+\sum_{i,k}\left(1-\frac{\deg\phi_i}{2}-|k|\right)\tau^{i,k}\phi_i\lambda^k,\] and $\mu_X\in\End_\C(H_T^*(X))$ is the \emph{grading operator} \[\mu_X(\phi_i\lambda^k)=\left(\frac{\deg\phi_i}{2}+|k|-\frac{\dim X}{4}\right)\phi_i\lambda^k.\]
	
	Define the $z$-sesquilinear pairing $P_X$ by \[P_X(f,g):=(f(-z),g(z))_X^T\in\C[z,\lambda]\fps{Q,\btau}.\]
	
	The module $H_T^*(X)[z]\fps{Q,\btau}$ equipped with the {equivariant quantum connection} $\nabla$ and the pairing $P_X$ is called the \emph{equivariant quantum D-module} of $X$, denoted by $\qdm_T(X)$. 
\end{definition} 

As before, the equivariant quantum connection $\nabla$ is flat, and compatible with $P_X$. Its \emph{fundamental solution} $M_X$ satisfies \begin{equation}\label{eq:fundsol}
	\begin{split}
		&M_X\circ\nabla_{\tau^{i,k}}=\partial_{\tau^{i,k}}\circ M_X\\
	&M_X\circ\nabla_{z\partial_z}=\left(z\partial_z-z^{-1}{c_1^T(X)\smile}+\mu_X\right)\circ M_X\\
	&M_X\circ\nabla_{\xi Q\partial_Q}=\left(\xi Q\partial_Q+z^{-1}{\xi\smile}+\mu_X\right)\circ M_X
	\end{split}
\end{equation} and can be expressed in terms of equivariant descendant invariants \begin{equation}\label{eq:fund}
	(M_X\phi_i,\phi_j)_X^T=(\phi_i,\phi_j)_X^T+\sum_{\substack{d\in H_2(X,\Z),n\in\Z\\(n,d)\ne(0,0)}}\left\langle\phi_i,\tau,\dots,\tau,\frac{\phi_j}{z-\psi}\right\rangle_{0,n+2,d}^{X,T}\frac{Q^d}{n!}.
\end{equation} By Theorem \ref{thm:loc}, $M_X$ is a formal power series in $Q,\tau$ with coefficients being rational functions in $\lambda,z$; more precisely, \[M_X\in\End_{\C[\lambda]}(H_T^*(X))\otimes_{\C[\lambda]}\C(\lambda,z)_{\hom}\fps{Q,\tau}\] where $\C(\lambda,z)_{\hom}:=\C(\lambda_1/z,\dots,\lambda_{\rank T}/z)[z,z^{-1}]$ (so inhomogeneous terms do not appear in denominators). The \emph{equivariant J-function} is $J_X=M_X\phi_0$, and has the explicit formula \begin{equation}\label{eq:j} J_X=1+\frac{\tau}{z}+\sum_{j=0}^s\sum_{\substack{d\in H_2(X,\Z),n\in\Z\\(n,d)\ne(0,0),(1,0)}}\phi^j\left\langle\tau,\dots,\tau,\frac{\phi_j}{z(z-\psi)}\right\rangle_{0,n+2,d}^{X,T}\frac{Q^d}{n!}\end{equation} where $(\phi_i,\phi^j)_X^T=\delta_i^j.$

We explain the Givental formalism in $T$-equivariant theory; the non-equivariant version can be recovered by setting $T=\{1\}.$
The \emph{Givental space} of $X$ is the infinite dimensional $\nov Q$-module \[\gsp_X:=H^*_T(X)\laur{z^{-1}}\fps{Q}.\] Following our earlier conventions, $H^*_T(X)\laur{z^{-1}}$ is the graded completion of $H^*_T(X)[z,z^{-1}]$ where $\deg z=2$. The Givental space is equipped with the symplectic form \[\Omega(\ff,\mathbf g):=-\res_{z=\infty}(\ff(-z),\mathbf g(z))_X^Tdz\] which takes values in $\C[\lambda]\fps{Q}.$ Note that $\res_{z=\infty}$ simply extracts the coefficient in front of the $z^1$ term. Consider the decomposition $\gsp_X=\gsp_+\oplus\gsp_-$ into maximally isotropic subspaces \[\gsp_+:=H^*_T(X)[z]\fps Q,\quad \gsp_-:=z^{-1}H^*_T(X)\fps{z^{-1}}\fps Q.\] The symplectic form $\Omega$ identifies $\gsp_-$ with the dual of $\gsp_+$ and $\gsp_X$ with the cotangent bundle $T^*\gsp_+$. The \emph{genus zero descendant Gromov--Witten potential} is given by \[\pot_X(z+\bt(z))=\sum_{\substack{n\in\N,d\in H_2(X,\Z)\\(n,d)\ne(0,0),(1,0),(2,0)}}\langle\bt(-\psi),\dots,\bt(-\psi)\rangle_{0,n,d}^{X,T}\frac{Q^d}{n!}\] where $\bt(z)=\sum_{n=0}^\infty t_nz^n$ and $t_n=\sum_{i=0}^st_n^i\phi_i\in H^*_T(X)$. Thus $\pot_X$ can be thought of as a function defined on the formal neighborhood of the point $z$ in $\gsp_+$.
The \emph{Givental cone} $\gcn_X\subset T^*\gsp_+\cong\gsp_X$ is the graph of the differential of $\pot_X$. Concretely, it consists of points \[\begin{split}
	\bigj_X(\bt)&=z+\bt(z)+\sum_{i=0}^s\sum_{k=0}^\infty(-1)^k\frac{\phi^i}{z^{k+1}}\frac{\partial\pot_X}{\partial t^i_k}\\
	&=z+\bt(z)+\sum_{i=0}^s\sum_{\substack{n\in\N,d\in H_2(X,\Z)\\(n,d)\ne(0,0),(1,0)}}\phi^i\left\langle\bt(-\psi),\dots,\bt(-\psi),\frac{\phi_i}{z-\psi}\right\rangle_{0,n+1,d}^{X,T}\frac{Q^d}{n!}
\end{split}\] where $\bt(z)\in\gsp_+$. By Theorem \ref{thm:loc}, $\bigj_X(\bt)$ lives in the \emph{rational Givental space} \[\gsp_X^\text{rat}:=H_{\wht T}^*(X)_\textrm{loc}\fps Q:=H_T^*(X)\otimes_{\C[\lambda]}\C(\lambda,z)_{\hom}\fps{Q}.\] The J-function \eqref{eq:j} induces a finite-dimensional family of points $zJ_X(\tau)$ in $\gcn_X$, which equals $\bigj_X(\bt)$ with $\bt=\tau.$ For this reason, $\bigj_X(\bt)$ is called the \emph{big J-function}.

We state the important geometric properties of the Givental cone (see \cite[Appendix B]{CCIT09} and \cite[Appendix A]{LSR25} for details). Let $T_\tau$ be the tangent space of $\gcn_X$ at $zJ_X(\tau)$. It is closed under multiplication by $z$ and has the structure of a $\C[\lambda,z]\fps Q$-module. In fact, $T_\tau$ is freely generated by the derivatives $z\partial_{\tau^i}J_X(\tau)=M_X(\tau)\phi_i$ over $\C[\lambda,z]\fps Q$, i.e., \[T_\tau=M_X(\tau)\gsp_+.\] The Givental cone is \emph{overruled}, namely, it can be written as the union of semi-infinite subspaces \[\gcn_X=\bigcup_{\tau\in H^*_T(X)\fps Q}zT_\tau,\] and $T_\tau$ is tangent to $\gcn_X$ exactly along $zT_\tau\subset\gcn_X.$ The above geometric properties precisely encode the dilaton equation, string equation, and topological recursion relations.

\subsection{Extended shift operators}

We introduce Iritani's \emph{extended shift operators} on the equivariant quantum D-module, in the symplectic context. These operators encode the Seidel representation \cite{Sei97} and the shift operators of Okounkov--Pandharipande \cite{OP10}.

Let $X$ be a symplectic manifold with a Hamiltonian $T$-action. Each cocharacter $k\in\Hom(S^1,T)\cong\Z^{\rank T}$ determines a loop of Hamiltonian symplectomorphisms. We consider the Hamiltonian fibration \[E_k \to \cp^1\] with fiber $X$, whose transition map between the two charts $(x,re^{iu})\in X\times \C^*\mapsto(y,se^{iv})\in X\times \C^*$ is given by \[y=\left(e^{iu}\right)^{-k}\cdot x,\quad se^{iv}=r^{-1}e^{-iu}.\] Let $H_2^\text{sec}(E_k,\Z)$ be the subgroup of section classes, i.e., homology classes in $H_2(E_k,\Z)$ mapping to the fundamental class $[\cp^1]$. For each $T$-fixed point $x$, there is an associated section class $\sigma_x(k)=x\times[\cp^1]$; obviously, the class only depends on the connected component.

Let $f_k:\cp^1\subset\cp^\infty=BS^1\to BT$ be the map induced by the homomorphism $k:S^1\to T$. It fits in the fiber square

\[\begin{tikzcd}
	{E_k} & {X\times_T ET} \\
	{\cp^1} & BT
	\arrow["{\td f_k}", from=1-1, to=1-2]
	\arrow[from=1-1, to=2-1]
	\arrow[from=1-2, to=2-2]
	\arrow["{f_k}", from=2-1, to=2-2]
\end{tikzcd}.\] Thus section classes of $E_k$ can be treated as equivariant homology classes of $X$: \[\td f_{k*}:H_2^\text{sec}(E_k,\Z)\hookrightarrow H_2^T(X,\Z).\]The image of $\td f_{k*}$ precisely consists of classes in $H_2^T(X,\Z)$ mapping to $f_{k*}[\cp^1]=k\in H_2^T(\pt,\Z)\cong\Hom(S^1,T).$ Thus we can view $H_2^T(X,\Z)$ as the disjoint union of $H_2^\text{sec}(E_k,\Z)$ for all $k\in\Hom(S^1,T)$.

We introduce an action of $\widehat{T}=T\times S^1$ on $E_k$ by \[(\lambda,z)\cdot(x,re^{iu})=(\lambda\cdot x,zre^{iu}),\quad (\lambda,z)\cdot(y,se^{iv})=\left(\lambda z^{-k}\cdot y,z^{-1}se^{iv}\right).\] The coordinates $\lambda,z$ will also be used to denote the corresponding equivariant parameters. Let $X_0,X_\infty$ denote the fibers of $E_k$ at $0,\infty\in\cp^1$ respectively. Then $\widehat{T}$ acts on $X_0$ by \[(\lambda,z)\cdot x=\lambda\cdot x,\ \textrm{ for } x\in X_0,\] and acts on $X_\infty$ by \[(\lambda,z)\cdot x=\lambda z^{-k}\cdot x,\ \textrm{ for } x\in X_\infty.\] The identity map $X_0\to X_\infty$ is equivariant with respect to the automorphism $\widehat T\to \widehat T,(\lambda,z)\mapsto(\lambda z^{k},z)$, hence it induces an isomorphism $\Phi_k:H^*_{\widehat T}(X_0)\to H^*_{\widehat T}(X_\infty)$ satisfying \[\Phi_k(f(\lambda,z)\alpha)=f(\lambda-kz,z)\Phi_k(\alpha)\] for $f(\lambda,z)\in H^*_{\widehat T}(\pt)\cong\C[\lambda,z]$ and $\alpha\in H_{\wht T}^*(X_0)$. Note that $H^*_{\widehat T}(X_0)=H^*_T(X)[z]$. For $\tau\in H^*_T(X)$, let $\hat\tau\in H^*_{\widehat T}(E_k)$ denote the unique class satisfying $\hat\tau|_{X_0}=\tau$ and $\hat\tau|_{X_\infty}=\Phi_k(\tau)$ (see \cite[Lemma 3.7]{Iri17}).

We choose an almost complex structure $\wht J$ on $E_k$ with the properties: \begin{itemize}
	\item The restriction of $\wht J$ to each fiber is $\omega$-tame.
	\item The projection map $\pi:E_k\to\cp^1$ is pseudo-holomorphic, i.e., $d\pi\circ \wht J=J_{\cp^1}\circ d\pi$, where $J_{\cp^1}$ is the standard complex structure.
	\item $\wht J$ is $\wht T$-invariant.
\end{itemize}
\begin{lemma}
	An almost complex structure satisfying those properties exists.
\end{lemma}
\begin{proof}
	We fix an almost complex structure $J$ on $X$ which is $\omega$-tame and $T$-invariant. The Hamiltonian fibration $E_k\to\cp^1$ has a splitting $H\oplus V$ into horizontal and vertical parts, determined by the symplectic connection. Let $J^v$ be the complex structure on $TV$ induced by $J$. Let $J^h$ be the complex structure on $TH$ pulled back from $J_{\cp^1}$. Then the almost complex structure $J^v\oplus J^h$ satisfies all the desired properties.
\end{proof}
Thus we may define the $\wht T$-equivariant Gromov--Witten invariants of $E_k$ associated to section classes. We may also define the monoid $\eff(E_k)$ of section classes with nonzero Gromov--Witten invariants. 
The Hamiltonian $S^1$-action on $X$ determined by $k$ has a unique fixed component $F_\textrm{max}(k)$ where the moment map attains its maximum. Let $\sigma_\textrm{max}(k)$ be the section class of $E_k$ associated to a fixed point in $F_\textrm{max}(k)$. 

\begin{lemma}[{\cite[Lemma 2.2]{GI12}, \cite[Lemma 3.6]{Iri17}}]\label{prop:efsec}
	$\eff(E_k)\subseteq\sigma_\textnormal{max}(k)+\eff(X).$
\end{lemma}
\begin{proof}
	By Theorem \ref{thm:loc}, it suffices to consider the classes of $\wht T$-invariant $J$-holomorphic curves. We also only consider section classes, so the only candidates are of the form $\sigma_x + d$, where $x$ is an $S^1$-fixed point of $X$, and $d$ is an effective class in $X_0\sqcup X_\infty.$ It remains to compare $\sigma_\textrm{max}$ and $\sigma_x$. Observe that $x$ can be joined to a point $x_\textrm{max}\in F_\textrm{max}(k)$ by a chain of gradient trajectories of the moment map $\mu$. We assume without loss of generality that $\mu(x)$ and $\mu(x_\textrm{max})$ are adjacent critical values, so the union of gradient trajectories joining $x$ and $x_\textrm{max}$ is symplectomorphic to $\cp^1$. The fibration of such $\cp^1$'s over the base $\cp^1$ is precisely the $k$th Hirzebruch surface. Its maximal and minimal section classes differ by a multiple of the fiber class, and the fiber class is effective. This concludes the proof.
\end{proof}

\begin{definition}\label{def:sei}
	For $\beta\in H_2^T(X,\Z)$, let $\overline\beta\in\Hom(S^1,T)$ be the image of $\beta$ under $H_2^T(X,\Z)\to H_2^T(\pt,\Z)$, and consider the fibration $E_k$ associated with $k:=-\overline\beta$. We define the $H^*_{\widehat T}(\pt)$-linear operator $\td{\mathbb{S}}^\beta:H^*_{\widehat T}(X_0)\fps{Q,\btau}\to Q^{\beta+\sigma_\textrm{max}(k)}H^*_{\widehat T}(X_\infty)\fps{Q,\btau}$ using the $\wht T$-equivariant Gromov--Witten invariants of $E_k$ as follows: \[(\td{\mathbb{S}}^\beta\alpha_0,\alpha_\infty)_{X_\infty}^{\wht T}=\sum_{n=0}^\infty\sum_{d\in H_2(X,\Z)}\langle i_{0*}\alpha_0,\hat\tau,\dots,\hat\tau,i_{\infty*}\alpha_\infty\rangle_{0,n+2,-\beta+d}^{E_k,\widehat T}\frac{Q^d}{n!}\] where $\alpha_0\in H_{\wht T}^*(X_0),\alpha_\infty\in H_{\wht T}^*(X_\infty)$, $i_0,i_\infty$ are the inclusion maps, and $-\beta+d$ is a section class in $E_k$. The \emph{(extended) shift operator} is defined to be \[\seiht^\beta:=\Phi_k^{-1}\circ\td{\mathbb{S}}^\beta:H_T^*(X)[z]\fps{Q,\btau}\to Q^{\beta+\sigma_\textrm{max}(k)}H_T^*(X)[z]\fps{Q,\btau}.\] 
\end{definition}
Shift operators satisfy the relation \begin{equation}\label{eq:shift}
	\seiht^\beta(f(\lambda,z)\alpha)=f(\lambda-\overline\beta z,z)\seiht^\beta\alpha
\end{equation} for $f(\lambda,z)\in H^*_{\widehat T}(\pt)=\C[\lambda,z].$ By dimension counting, we see that $\seiht^\beta$ is homogeneous of degree $2c_1^T(X)\cdot\beta.$ If $\beta\in H_2(X,\Z)\subset H_2^T(X,\Z)$, then $\seiht^\beta$ is simply multiplication by the Novikov variable $Q^\beta$.

We also define shift operators on the {rational Givental space} $\gsp^\text{rat}_X$.

\begin{definition}
	For a fixed component $F$, let $\nm_F$ be its normal bundle in $X$, with weight decomposition $\nm_F=\bigoplus_{\alpha}\nm_{F,\alpha}$ so that $T$ acts on $\nm_{F,\alpha}$ by the character $\alpha\in\Hom(T,S^1).$ Let $\{\rho_{F,\alpha,j}\}_{1\le j\le\rank \nm_{F,\alpha}}$ be the Chern roots of $\nm_{F,\alpha}$. For $\beta\in H_2^T(X,\Z)$, we define the operator $\sftht^\beta:\gsp_X^\text{rat}\to Q^{\beta+\sigma_\textrm{max}(-\overline\beta)}\gsp_X^\text{rat}$ using the localization isomorphism $H^*_{\widehat T}(X)_\text{loc}\cong\bigoplus_F H^*_{\widehat T}(F)_\text{loc}$ as follows: \[
\sftht^\beta \mathbf{f}|_{F}
=
Q^{\beta + \sigma_F(-\overline\beta)}
\left(
\prod_{\alpha}
\prod_{j=1}^{\mathrm{rank}\nm_{F,\alpha}}
\frac{
\prod_{m=-\infty}^{0} \left( \rho_{F,\alpha,j} + \alpha + m z \right)
}{
\prod_{m=-\infty}^{-\alpha\cdot \overline\beta} 
\left( \rho_{F,\alpha,j} + \alpha + m z \right)
}
\right)
e^{-z\overline\beta \partial_{\lambda}}
(\mathbf{f}|_{F})
\] where $\sigma_F(-\overline\beta)\in H_2^\text{sec}(E_{-\overline\beta},\Z)$ is the section class associated with a fixed point in $F$, the character $\alpha$ is identified with an element of $H_T^2(\pt)$, and $e^{-z\overline\beta \partial_{\lambda}}$ is the shift $f(\lambda,z)\mapsto f(\lambda-\overline\beta z,z)$ of equivariant parameters acting on $\C(\lambda,z)_\text{hom}$.
\end{definition}

The fundamental solution $M_X$ also trivializes shift operators, in the following sense:

\begin{theorem}[{\cite[Theorem 3.14]{Iri17}}]
	Let $M_X$ be the fundamental solution in $T$-equivariant theory. Then \begin{equation}\label{eq:fsolshft}
		\sftht^\beta\circ M_X=M_X\circ\seiht^\beta.
	\end{equation}
\end{theorem}
\begin{proof}
	We compute $\seiht^\beta$ using Theorem \ref{thm:loc}. Write $k=-\overline\beta$. A $\wht T$-fixed stable map $f:(C,x_1,\dots,x_{n+2})\to E_k$ must have the following form: \begin{itemize}
		\item $C=C_0\cup C_\textrm{sec}\cup C_\infty$ where $C_\textrm{sec}\cong\cp^1$,
		\item $f_0=f|_{C_0}$ is a $T$-fixed stable map to $X_0$,
		\item $f_\infty=f|_{C_\infty}$ is a $T$-fixed stable map to $X_\infty$,
		\item $f_\textrm{sec}=f|_{C_\textrm{sec}}$ is a section of $E_k$ associated to a $T$-fixed point in $X$.
	\end{itemize} In other words, the $\wht T$-fixed locus in the moduli space $(E_k)_{0,n+2,d}$ is given by \[\bigsqcup_i\bigsqcup_{I_1\sqcup I_2=\{1,\dots,n+2\}}\bigsqcup_{d_0+d_\infty+\sigma_i=d}((X_0)_{0,I_1\cup p,d_0})^T\times_{F_i}((X_\infty)_{0,I_2\cup q,d_\infty})^T\] where $\{F_i\}$ are the $T$-fixed components of $X$, and $\{\sigma_i\}$ are the corresponding section classes. The virtual normal bundle of $((X_0)_{0,I_1\cup p,d_0})^T\times_{F_i}((X_\infty)_{0,I_2\cup q,d_\infty})^T$ is \[\vnm_0+\vnm_\infty+\vnm_{\textrm{sec},i}-N_{F_i/X_0}-N_{F_i/X_\infty}+\mathcal{L}_p^{-1}\otimes\xi+\mathcal{L}_q^{-1}\otimes\xi^{-1}\] where $\vnm_0$ is the virtual normal bundle of $((X_0)_{0,I_1\cup p,d_0})^T$ in $(X_0)_{0,I_1\cup p,d_0}$, $\vnm_\infty$ is the virtual normal bundle of $((X_\infty)_{0,I_2\cup q,d_\infty})^T$ in $(X_\infty)_{0,I_2\cup q,d_\infty}$, $\vnm_{\textrm{sec},i}$ is the virtual normal bundle given by deforming $f_{\textrm{sec},i}$, $\mathcal L_p$ (resp. $\mathcal L_q$) is the universal cotangent line bundle at $p$ (resp. $q$), and $\xi$ is the trivial line bundle with weight 1.  
	
	By Theorem \ref{thm:loc}, we have \begin{multline*}
		(\td\sei^\beta(\tau)\alpha_0,\alpha_\infty)=\sum_{i,l,m,a,b}\sum_{d_0+d_\infty+\sigma_i=d}\left\langle z\alpha_0,\tau,\dots,\tau,\frac{(\iota_{i,0})_*\phi_{i,a}}{z-\psi}\right\rangle_{0,l+2,d_0}^{X_0,\wht T}\frac{Q^{d_0}}{l!}\\
		\left(\int_{F_i}\frac{Q^{\beta+\sigma_i}}{e_{\wht T}(\vnm_{\textrm{sec},i})}\phi_i^a\phi_i^b\right)\left\langle\frac{(\iota_{i,\infty})_*\phi_{i,b}}{-z-\psi},\tau',\dots,\tau',-z\alpha_\infty\right\rangle_{0,m+2,d_\infty}^{X_\infty,\wht T}\frac{Q^{d_\infty}}{m!}
	\end{multline*} where $\iota_{i,0}:F_i\to X_0$ and $\iota_{i,\infty}:F_i\to X_\infty$ are the inclusions, $\tau'=\Phi_k(\tau)$, $\{\phi_{i,a}\}_a$ is a basis of $H^*(F_i)$, and $\{\phi_i^a\}_a$ is the dual basis. Indeed, recall that $z$ is the equivariant parameter of the $S^1$-action, so $z\alpha_0, -z\alpha_\infty$ are the push-forwards; the $N_{F_i/X_0},N_{F_i/X_\infty}$ contributions are absorbed in $(\iota_{i,0})_*,(\iota_{i,\infty})_*$; the $\vnm_0,\vnm_\infty$ contributions are absorbed in the equivariant Gromov--Witten invariants of $X_0,X_\infty$ by Theorem \ref{thm:loc}; the $\mathcal L_p^{-1}\otimes\xi,\mathcal L_q^{-1}\otimes\xi^{-1}$ contributions of the smoothing of nodes gives the factors $1/(z-\psi)$ and $1/(-z-\psi).$
	
	Let $N_{F_i/X}=\bigoplus_\alpha N_{i,\alpha}$ be the weight decomposition. Then the normal bundle of $F_i\times\cp^1$ in $E_k$ is \[\bigoplus_\alpha N_{i,\alpha}\boxtimes\mathcal O_{\cp^1}(\alpha\cdot k).\] Therefore \begin{equation*}
		\vnm_{\textrm{sec},i}=\xi\oplus\xi^{-1}\oplus\bigoplus_\alpha N_{i,\alpha}\otimes\left(\bigoplus_{j\le0}\xi^j-\bigoplus_{j<-\alpha\cdot k}\xi^j\right),
	\end{equation*} where $\xi\oplus\xi^{-1}$ comes from the base $\cp^1$ and $\bigoplus_{j\le0}\xi^j-\bigoplus_{j<-\alpha\cdot k}\xi^j$ comes from $\mathcal O_{\cp^1}(\alpha\cdot k)$. Then we find that \[\frac{Q^{\beta+\sigma_i}}{e_{\wht T}(\vnm_{\textrm{sec},i})}=\frac{1}{z(-z)}\frac{e^{-kz\partial_\lambda}\Delta_i(k)}{e_{\wht T}(N_{F_i/X_\infty})}\] where $\Delta_i(k)$ is precisely the expression such that $\sftht^\beta|_{F_i}=\Delta_i(k)e^{k z\partial_\lambda}$. We conclude that \[(\td\sei^\beta(\tau)\alpha_0,\alpha_\infty)=(\td\sft^\beta(\tau)M_X(\tau,z)\alpha_0,M_X'(\tau',-z)\alpha_\infty),\] where $\td\sft^\beta(\tau)$ is $\sftht^\beta(\tau)$ with $\Delta_i(k)e^{k z\partial_\lambda}$ replaced by $e^{-kz\partial_\lambda}\Delta_i(k)$, and $M_X'(\tau',z)$ is $M_X(\tau,z)$ with $T$-equivariant Gromov--Witten invariants of $X_0$ replaced by $\wht T$-equivariant Gromov--Witten invariants of $X_\infty$. Note that the factor $1/e_{\wht T}(N_{F_i/X_\infty})$ gets absorbed when switching from the pairing on $F_i$ to the pairing on $X_\infty$. Equation \eqref{eq:fsolshft} follows by noting that $M_X'(\tau',z)=\Phi_k\circ M_X(\tau,z)\circ\Phi_k^{-1},\ \td\sft^\beta=\Phi_k\circ\sftht^\beta$, and the unitarity property $M_X(\tau,-z)^*=M_X(\tau,z)^{-1}.$ Recall that the unitarity of $M_X$ follows from the quantum differential equation and the Frobenius property of the quantum product; in some literature this is part of the definition of $M_X$ (as an element of the twisted loop group).
\end{proof}
It follows that the tangent space $T_\tau$ of the Givental cone $\gcn_X$ is also invariant under shift operators: \[\sftht^\beta T_\tau\subset Q^{\beta+\sigma_\textrm{max}(-\overline\beta)}T_\tau,\quad \textrm{ for }\beta\in H_2^T(X,\Z).\]
Another important consequence is the following:

\begin{corollary}[{\cite[Corollary 2.11]{Iri25}}]\label{cor:shftcomm}
	Shift operators define a representation of $H_2^T(X,\Z)$, i.e., \[\seiht^0=\id,\quad \seiht^{\beta_1+\beta_2}=\seiht^{\beta_1}\circ\seiht^{\beta_2}.\] Also, shift operators are compatible with the quantum connection and the pairing: \begin{align*}
		&[\nabla_{\tau^{i,k}},\seiht^\beta]=0,\quad [\nabla_{z\partial_z},\seiht^\beta]=0,\quad [\nabla_{\xi Q\partial_Q},\seiht^\beta]=(\xi\cdot\beta)\seiht^\beta,\\
		&P_X(\seiht^{-\beta}f,\seiht^\beta g)=e^{-z\overline\beta\partial_\lambda}P_X(f,g).
	\end{align*}
\end{corollary}

\subsection{Twisted Gromov--Witten invariants}
Let $V\to X$ be a Hermitian vector bundle. Let $\cc$ be an invertible multiplicative characteristic class. 

\begin{definition}
	Let $X_{0,n,d}$ be the moduli space as before. Consider the diagram 
\[\begin{tikzcd}
	{X_{0,n+1,d}} & X \\
	{X_{0,n,d}}
	\arrow["\ev_{n+1}", from=1-1, to=1-2]
	\arrow["f"', from=1-1, to=2-1]
\end{tikzcd}\] where $f$ is the forgetful map of the $(n+1)$th marked point. We define the virtual bundle \[V_{0,n,d}:=f_*\ev_{n+1}^*V\in K(X_{0,n,d})\] where $f_*$ is the K-theoretic push-forward. Its fiber at the stable map $u:C\to X$ is \[H^0(C,u^*V)\ominus H^1(C,u^*V).\] The \emph{$(V,\cc)$-twisted Gromov--Witten invariants} are defined to be \[\langle\alpha_1\psi^{k_1},\dots,\alpha_n\psi^{k_n}\rangle_{0,n,d}^{X,(V,\cc)}:=\int_{[X_{0,n,d}]^\textrm{vir}}\prod_{i=1}^n\ev_i^*(\alpha_i)\psi_i^{k_i}\cup\cc(V_{0,n,d}).\] 

For the details of this definition, see Appendix \ref{sec:app}.
\end{definition}

We introduce the \emph{twisted Poincar\'e pairing} \[(\alpha_1,\alpha_2)_X^{(V,\cc)}:=\int_X\alpha_1\cup\alpha_2\cup\cc(V).\] Using the twisted Poincar\'e pairing in place of the untwisted one, we define twisted versions of the \emph{fundamental solution} and \emph{J-function} in the same way as before.

We also have Givental cones for twisted Gromov--Witten invariants. We specialize to the following two cases: $\cc=e_\lambda^{-1}$ and $\cc=\td e_\lambda^{-1}$, where \[e_\lambda(V)=\sum_{i=0}^{\rank V}c_i(V)\lambda^{\rank V-i},\quad \td e_\lambda(V)=\sum_{i=0}^{\rank V}c_i(V)\lambda^{-i}.\] Here $\lambda$ can be thought of as the equivariant parameter of the fiberwise $\C^*$-action on $V$. The degree of the Novikov variable $Q^d$ is $2c_1(X)\cdot d$ in the $(V,\td e_\lambda^{-1})$-twisted theory and $2(c_1(X)+c_1(V))\cdot d$ in the $(V, e_\lambda^{-1})$-twisted theory. We introduce coefficient rings $R=\C[\lambda,\lambda^{-1}]$ for $\cc=e_\lambda^{-1}$ and $R=\C[\lambda^{-1}]$ for $\cc=\td e_\lambda^{-1}$. The Givental space for the twisted theory is \[\gsp_X^\textrm{tw}:=H^*(X)\otimes R\laur z\fps Q.\] Note that here we allow power series of $z$ infinite in the positive direction, as opposed to the untwisted version; this is for the quantum Riemann--Roch operator \eqref{eq:qrrop} to be defined. The Givental space $\gsp_X^\textrm{tw}$ is equipped with the symplectic form \[\Omega_\textrm{tw}(\ff,\mathbf g):=\res_{z=0}(\ff(-z),\mathbf g(z))_X^{(V,\cc)}dz.\] It admits a decomposition $\gsp_X^\textrm{tw}=\gsp_+^\textrm{tw}\oplus\gsp_-^\textrm{tw}$ where \[\gsp_+^\textrm{tw}=H^*(X)\otimes R\fps{z}\fps{Q},\quad \gsp_-^\textrm{tw}=z^{-1}H^*(X)\otimes R[z^{-1}]\fps{Q},\] and we identify $\gsp_X^\textrm{tw}$ with $T^*\gsp_+^\textrm{tw}$. The \emph{twisted Givental cone} $\gcn_X^\textrm{tw}$ is defined to be the graph of the differential of the \emph{twisted descendant Gromov--Witten potential}. It consists of points \[
	\bigj_X^\textrm{tw}(\bt)
	=z+\bt(z)+\sum_{i=0}^s\sum_{\substack{n\in\N,d\in H_2(X,\Z)\\(n,d)\ne(1,0)}}\frac{\phi^i}{\cc(V)}\left\langle\bt(-\psi),\dots,\bt(-\psi),\frac{\phi_i}{z-\psi}\right\rangle_{0,n+1,d}^{X,(V,\cc)}\frac{Q^d}{n!}\] where $\bt(z)\in\gsp_+^\textrm{tw}$. This is contained in $\gsp_X^\textrm{tw}$ since the $\psi$-class is nilpotent.

Coates--Givental proved that twisted Gromov--Witten invariants are determined by untwisted Gromov--Witten invariants \cite{CG07}. Let $\mathbf s=(s_0,s_1,s_2,\dots)$ be variables such that $\cc(\cdot)=\exp\left(\sum_{k=0}^\infty s_k\mathrm{ch}_k(\cdot)\right)$, where $\mathrm{ch}_k$ is the $k$th component of the Chern character. We set $\deg s_k=-2k$. In this case the Givental space can be written as \[\gsp_X^\textrm{tw}=H^*(X)[z,z^{-1}]\fps{Q,\mathbf s}.\] The \emph{quantum Riemann--Roch} operator $\Delta_{(V,\cc)}:\gsp_X^\textrm{tw}\to\gsp_X^\textrm{tw}$ is defined to be \begin{equation}\label{eq:qrrop}
	\Delta_{(V,\cc)}=\exp\left(\sum_{l,m\ge0,l+m\ge1}s_{l+m-1}\frac{B_m}{m!}\mathrm{ch}_l(V)(-z)^{m-1}\right)
\end{equation} where $B_m$ are the Bernoulli numbers given by $\sum_{m=0}^\infty\frac{B_m}{m!}x^m=\frac{x}{e^x-1}.$ It is the asymptotic expansion of the infinite product $\prod_m\cc(V\otimes L^m)$ where $L$ is a line bundle whose first chern class is $z$.
\begin{theorem}\label{thm:qrr}
	The $(V,\cc)$-twisted Givental cone is $\Delta_{(V,\cc)}\gcn_X$, where $\gcn_X$ is the untwisted Givental cone defined within $\gsp_X^\textnormal{tw}$.
	
	More generally, if $V$ has a fiberwise $T$-action, let $V_\alpha$ be the $T$-eigenbundles, with classes $\cc_\alpha$. Then the $(V,\cc_T)$-twisted Givental cone is $\left(\prod_\alpha\Delta_{(V_\alpha,\cc_\alpha)}\right)\gcn_X.$
\end{theorem}

The first statement is \cite[Corollary 4]{CG07} (see also \cite[Remark 2.15]{Iri25}). The second, slightly more general statement, is \cite[Theorem 1.1]{Ton14}. We will prove this in the symplectic setting in Appendix \ref{sec:app}, based on ideas of Coates \cite{Coa03}.

\section{Decomposition theorem}\label{sec:decomp}
In this section we give the proof of Theorem \ref{thm:blowup}. It is essentially the same as \cite{Iri25}. Therefore we will only sketch the proof, only expanding in detail when we feel it is necessary to do so. For further details the reader should refer to \cite{Iri25}.

Let us explain the rough idea of the proof. Let $\td X$ be the symplectic blow-up of $X$ along $Z$. Both $X$ and $\td X$ can be realized as the symplectic reductions of a symplectic manifold $W$ by a Hamiltonian $S^1$-action (see Proposition \ref{prop:master}). The space $W$ also contains $Z$ as one of its fixed components. Shift operators act on $\qdm_{S^1}(W)$, and the key idea is to consider the adic completion of $\qdm_{S^1}(W)$ by a certain ideal of shift operators. Denote this completion by $\qdm_{S^1}(W)\sphat_{\td X}$. We then construct the following isomorphisms: 
\[\begin{tikzcd}[sep=small]
	& {\qdm_{S^1}(W)\sphat_{\td X}} & \\
	{\qdm(\td X)} && {\qdm(X)\oplus\qdm(Z)^{\oplus(r-1)}}
	\arrow["\cong"', from=1-2, to=2-1]
	\arrow["\cong", from=1-2, to=2-3]
\end{tikzcd}\textrm{modulo pullbacks and extensions.}\]
The maps are constructed using the virtual localization formula (Theorem \ref{thm:loc}) and quantum Riemann--Roch theorem (Theorem \ref{thm:qrr}). For the maps to be defined over the completion $\qdm_{S^1}(W)\sphat_{\td X}$, we must ensure compatibility with shift operators. For this reason, the maps are in fact defined by ``Fourier transforms''. To prove that they are indeed isomorphisms, we first show that $\qdm_{S^1}(W)\sphat_{\td X}$ is a free finite module of the same rank as $\qdm(\td X)$, using the fact that the equivariant fundamental solution is a solution of the shift operators, and properties of the $S^1$-equivariant cohomology of $W$ and the Kirwan maps. Then we explicitly compute the leading order terms.

\subsection{Symplectic reductions}
We describe the relation between symplectic blow-ups and symplectic reductions. This has been done in \cite{GS89}.

\begin{proposition}[{\cite[Section 12]{GS89}}]\label{prop:master}
	One can construct a symplectic manifold $W$ with a Hamiltonian $S^1$-action, with moment map $\mu:W\to\R$, such that the following holds: \begin{enumerate}[(1)]
		\item The image $\mu(W)$ is a closed interval centered at $0$.
		\item For $t>0$, the symplectic reduction $\mu^{-1}(t)/S^1$ is symplectomorphic to $(X,\omega)$.
		\item $0$ is a critical value of $\mu$.
		\item For $-t<0$, the symplectic reduction $\mu^{-1}(-t)/S^1$ is symplectomorphic to the symplectic blow-up $(\td X,\td\omega_t)$ of $X$ along $Z$ with blow-up parameter $t$.
	\end{enumerate}
\end{proposition}
The symplectic manifold $W$ can be constructed as the symplectic blow-up of $X\times \cp^1$ along $Z\times\{0\}$. Write $T=S^1$. The Hamiltonian $T$-action on $W$ is induced by the $T$-action on $X\times\cp^1$ by rotating the $\cp^1$ component, where $0,\infty\in\cp^1$ are the two fixed points, and the action has weight 1 near 0 and weight $-1$ near $\infty$. The fixed components of $W$ are \[\mu^{-1}(\max)\cong X,\ \mu^{-1}(\min)\cong\td X,\ \mu^{-1}(0)\cong Z.\]
We have \[H_2^{T}(W,\Z)\cong H_2(W,\Z)\oplus\Z\cong H_2(X,\Z)\oplus\Z^3.\] An element $(d,k,l,m)$ in the right hand side corresponds to the unique class $\beta\in H_2^{T}(W,\Z)$ such that \[\pr_{1*}\hat\varphi_*\beta=(d,m)\in H_2^T(X)\cong H_2(X)\oplus\Z,\quad [X]\cdot\beta=k,\quad -[\wht D]\cdot\beta=l.\] Here $\hat\varphi:W\to X\times\cp^1$ is the blow-down map, $\wht D=\hat\varphi^{-1}(Z\times\{0\})$ the exceptional divisor, $[X],[\wht D]$ the Poincar\'e duals of $X\times\{\infty\}$, $\wht D$ in $W$ respectively. The subgroup $H_2(W,\Z)$ constists of elements of the form $(d,k,l,0)$.

We choose an almost complex structure $\td J$ on $W$ with the properties: \begin{itemize}
	\item $\td J$ is tamed by the blow-up symplectic form.
	\item $\hat\varphi$ is pseudo-holomorphic, i.e., $d\hat\varphi\circ \td J=(J,J_{\cp^1})\circ d\hat\varphi$ for some $J$ tamed by $\omega$ and $J_{\cp^1}$ is the standard complex structure on $\cp^1$.
	\item $\td J$ is $T$-invariant.
\end{itemize} 
\begin{lemma}
	An almost complex structure satisfying those properties exists.
\end{lemma}
\begin{proof}
	We shall freely use notation from Section \ref{sec:blowup}. Write $G=\U(r).$ The normal bundle of $Z$ in $X$ is identified with $P\times_G\C^{r}$, where $P$ is the associated principal $G$-bundle. Let $\td U$ be the product neighborhood of $P\times_G\td B(\epsilon_0)$ with a small disc centered at $0\in\cp^1$. The connection $A$ on $P$ determines a splitting $T\td U=\td H\oplus\td V$ into horizontal and vertical subbundles. 
	
	We fix a tame almost complex structure $J$ on $X$ which preserves the tangent bundle $TZ$ and the symplectic normal bundle $NZ$. Let $J^h$ be the complex structure on $\td H$ pulled back from $TZ$. Let $J^v$ be the complex structure on $\td V$ induced by the standard holomorphic structure on $\td B(\epsilon_0)$ and the small disc in $\cp^1$. We form the almost complex structure $J^h\oplus J^v$ on $\td U$. Outside the $\epsilon$-neighborhood of $Z\times\{0\}$, we take the pull-back of $J\oplus J_{\cp^1}$. It agrees with $J^h\oplus J^v$ on the common boundary, since the same property holds for the blow-up $\td\C^{r}_\epsilon$. Thus we have defined an almost complex structure $\td J$ on $W$. 
	
	The only nontrivial property to verify is that $\td J$ is tamed by the blow-up symplectic form. By construction, $\td H$ and $\td V$ are orthogonal, so it suffices to show that $J^h$ and $J^v$ are tame, individually. On each fiber the symplectic form $\td\omega_\epsilon$ tames $J^v$. On the horizontal subbundle the symplectic form is $\omega_Z+\td\alpha$, where $\td\alpha$ is determined by the moment map on $\td\C^{r}_\epsilon$ and the connection $A$. If $\epsilon$ is sufficiently small, then $\td\alpha$ is very small compared to $\omega_Z$. Since $\omega_Z$ tames $J^h$, so does $\omega_Z+\td\alpha$.
\end{proof}
Thus we may define the $T$-equivariant Gromov--Witten invariants of $W$, and the monoid of effective classes $\eff(W)$. By Theorem \ref{thm:loc}, the classes in $\eff(W)$ are no more than the classes of $T$-invariant curves, so we obtain the following:

\begin{lemma}[{\cite[Lemma 3.1]{Iri25}}]\label{lem:wclass}
	The monoid $\eff(W)$ is generated by $(d,0,0,0)$ for $d\in\eff(X)$, $(\varphi_*\td d,0,-[D]\cdot\td d,0)$ for $\td d\in\eff(\td X)$, and $(0,1,-1,0)$. Here $\varphi:\td X\to X$ is the blow-down map and $D$ is its exceptional divisor.
\end{lemma}

Let $Q^d,x,y,S$ denote the elements in the group ring $\C[H_2^T(W,\Z)]$ corresponding respectively to $(d,0,0,0),(0,1,0,0),(0,0,1,0),(0,0,0,1).$ The Novikov variables are collectively denoted by $\sq=(Q,x,y)$, and we denote $\wht S=(\sq,S).$ We have $\deg Q^d=2c_1(X)\cdot d$, $\deg x=4$, $\deg y=2r$, $\deg S=2$. The Novikov ring of $W$ is $\nov\sq$.
We also introduce the monoids \begin{align*}
	&C_X=\eff(W)\oplus\N(0,1,0,-1)\oplus\N(0,0,-1,1),\\
	&C_{\td X}=\eff(W)\oplus\N(0,0,1,-1)\oplus\N(0,0,0,1),\\
	&C^T(W)=C_X\cap C_{\td X}=\eff(W)\oplus\N(0,1,0,-1)\oplus\N(0,0,0,1),
\end{align*} which will be useful later.

We recall the additive structure of $H_T^*(W)$. The arguments of \cite[Proposition 2.4]{McD84} clearly generalize to $T$-equivariant cohomology, so we obtain the following: Let $\hat\jmath:\wht D\to W$ be the inclusion, $\hat\pi:\wht D\to Z$ the projection, and $\hat p=-\hat\jmath^*[\wht D]$. Then \begin{equation*}
	H_T^*(W)=\hat\varphi^*H_T^*(X\times\cp^1)\oplus\bigoplus_{k=0}^{r-1}\hat\jmath_*(\hat{p}^k\hat\pi^*H_T^*(Z)),
\end{equation*} where $H_T^*(X\times\cp^1)\cong H^*(X)[[0],[\infty]]/([0]\cdot[\infty])$ with $[0]-[\infty]=\lambda$, and $H_T^*(Z)=H^*(Z)[\lambda].$
The restrictions of $f=\hat\varphi^*\alpha+\sum_{k=0}^{r-1}\hat\jmath_*(\hat{p}^k\hat\pi^*\gamma_k)\in H_T^*(W)$ to the fixed loci $X,Z,\td X$ are \begin{equation}\label{eq:abloc}
	\begin{split}
	f_X&=\alpha|_{X\times\{\infty\}},\\
	f_Z&=\alpha|_{Z\times\{0\}}-\sum_{k=0}^{r-1}(-\lambda)^{k+1}\gamma_k,\\
	f_{\td X}&=\varphi^*(\alpha|_{X\times\{0\}})+\sum_{k=0}^{r-1}\jmath_*({p}^k\pi^*\gamma_k)\\
	&=\varphi^*(\alpha|_{X\times\{0\}}+\iota_*\gamma_{r-1})+\sum_{k=0}^{r-2}\jmath_*({p}^k\pi^*(\gamma_k-c_{r-1-k}(\nm_{Z/X})\gamma_{r-1}))
\end{split}
\end{equation} where $\varphi,\jmath,\pi,p,\iota$ are the same as in Section \ref{sec:blowup}. The formula for $f_Z$ uses $[\wht D]|_Z=\lambda$ and $\hat p|_Z=-\lambda.$ The second formula for $f_{\td X}$ uses $\pi_*(p^{r-1})=1$ and the relation \begin{equation}\label{eq:projbdrel}
	p^r+p^{r-1}\pi^*c_1(\nm_{Z/X})+\dots+\pi^*c_r(\nm_{Z/X})=0
\end{equation} in $H^*(D)$; the relation \eqref{eq:projbdrel} appears in the proof of \cite[Proposition 2.5]{McD84}.

The normal bundles of the fixed components, with $T$-action, are: \begin{equation}\label{eq:nm}
	\begin{split}
		&\nm_{X/W}\cong\underline\C\quad \textrm{ with weight } {-1},\\
		&\nm_{\td X/W}\cong\underline\C\otimes\mathcal{O}_{\td X}(-D)\quad \textrm{ with weight } 1+0=1,\\
		&\nm_{Z/W}\cong (\nm_{Z/X}\otimes\underline\C)\oplus\underline\C\quad \textrm{ with weight } (0-1=-1,1).
	\end{split}
\end{equation}

Using \eqref{eq:abloc}, \eqref{eq:projbdrel}, \eqref{eq:nm}, and the localization formula \cite{AB84}, we obtain:
\begin{lemma}[{\cite[Lemma 3.7]{Iri25}}]\label{lem:abloc}
	Let $f\in H_T^*(W).$ \begin{enumerate}[(1)]
		\item If $f_X=0$, then $f_Z$ is divisible by $\lambda$, and $f=\hat\jmath_*(\hat\pi^*f_Z/\lambda)+i_{\td X*}g$ for some $g\in H_T^*(\td X)$.
		\item If $f_{\td X}=0$, then $f_Z$ is divisible by $e_{-\lambda}(\nm_{Z/X})$, and $f=i_{Z\times\cp^1*}(\pr_1^*f_Z/e_{-\lambda}(\nm_{Z/X}))+i_{X*}g$ for some $g\in H_T^*(X)$. Here $e_{-\lambda}$ is the equivariant Euler class $e_\lambda$ with $\lambda$ replaced by $-\lambda$. 
	\end{enumerate}
\end{lemma}

We next describe the \emph{Kirwan maps} $\kappa_X:H^*_{T}(W)\to H^*(X)$ and $\kappa_{\td X}:H^*_{T}(W)\to H^*(\td X)$. For $\alpha\in H^*_{T}(W)$, $\kappa_X(\alpha)$ is obtained by restricting $\alpha$ to $X$ and setting $\lambda=0$; and $\kappa_{\td X}(\alpha)$ is obtained by restricting $\alpha$ to $\td X$ and setting $\lambda=[D]$. Kirwan maps are surjective \cite{Kir84}. 
By Lemma \ref{lem:abloc}, we deduce:
\begin{lemma}[{\cite[Lemma 3.8]{Iri25}}]\label{lem:kerkir}
	$\ker\kappa_{\td X}=i_{\td X*}H_T^*(\td X)+i_{X*}H_T^*(X)+i_{Z\times\cp^1*}(\pr_1^*H_T^*(Z)).$
\end{lemma} 

The \emph{dual Kirwan map} $\kappa_X^*:H_2(X,\Z)\to H_2^T(W,\Z)$ sends $d$ to $(d,0,0,0)$, and $\kappa_{\td X}^*:H_2(\td X,\Z)\to H_2^T(W,\Z)$ sends $\td d$ to $(\varphi_*\td d,0,-[D]\cdot\td d,[D]\cdot\td d)$. It follows from Lemma \ref{lem:wclass} that $\kappa_X^*$ embeds $\eff(X)$ into $C_X$ and $\kappa_{\td X}^*$ embeds $\eff(\td X)$ into $C_{\td X}$. Hence the Novikov rings $\nov Q,\nov{\td Q}$ of $X,\td X$ can be identified with subrings of $\nov{C_X},\nov{C_{\td X}}$ respectively. Explicitly, $\td Q^{\td d}$ is identified with $Q^{\varphi_*\td d}(y^{-1}S)^{[D]\cdot\td d}=\sq^{i_*\td d}S^{[D]\cdot\td d}$, where $i:\td X\to W$ is the inclusion. These identifications are compatible with the grading.

We write $\theta\in H^*_{T}(W)$, and let $\bth=\{\theta^{i,k}\}$ denote the infinite set of variables dual to a $\C$-basis $\{\phi_i\lambda^k\}$ of $H^*_{T}(W)$. Then the equivariant quantum D-module of $W$ is $\qdm_T(W)=H_T^*(W)[z]\fps{\sq,\bth}$ equipped with the equivariant quantum connection $\nabla$ and pairing $P_W$. It is also equipped with the action of shift operators $\seiht^\beta$ for $\beta\in H_2^T(W,\Z).$ We write $\mathbb S:=\seiht^{(0,0,0,1)}.$ It follows from the definition of shift operators that for general $\beta=(d,k,l,m)\in H_2^T(W,\Z)$, we have $\seiht^\beta=Q^dx^ky^l\mathbb S^m.$
Using $[X]|_X=-\lambda$ and $[\wht D]|_Z=\lambda$, we find that the section class $\sigma_F(k)\in H_2^\textrm{sec}(E_k(W),\Z)$ associated to each fixed component $F=X,Z,\td X$ is \[\sigma_X(k)=(0,-k,0,k),\ \sigma_Z(k)=(0,0,-k,k),\ \sigma_{\td X}(k)=(0,0,0,k).\] The maximal component $F_\textrm{max}(k)$ is $X$ if $k>0$, $W$ if $k=0$, and $\td X$ if $k<0$. Thus the corresponding section class $\sigma_\textrm{max}(k)=(0,\min(-k,0),0,k)$, and hence the shift operator defines a map $\sei^k:\qdm_T(W)\to x^{\min(k,0)}\qdm_T(W)$.

\begin{proposition}[{\cite[Proposition 3.9]{Iri25}}]\label{prop:shftact}
	The equivariant quantum D-module $\qdm_T(W)$ has the structure of a $\C[z]\fps{C^T(W)}$-module, where $\C[z]\fps{C^T(W)}$, as a $\C[z]$-algebra, is topologically generated by $\sq^\delta$ for $\delta\in\eff(W)$, $\mathbb S$, and $x\mathbb S^{-1}$.
\end{proposition}

We also write $\sft:=\sftht^{(0,0,0,1)}:\gsp_W^\text{rat}\to \gsp_W^\text{rat}$ for the shift operator on the rational Givental space. 
Unpacking the definitions we see that \begin{equation}\label{eq:shft} \begin{split}
	&(\sft^k\ff)_X=x^k\frac{\prod_{m=-\infty}^0-\lambda+mz}{\prod_{m=-\infty}^k-\lambda+mz}e^{-kz\partial_\lambda}\ff_X,\\
	&(\sft^k\ff)_Z=y^k\frac{\prod_{m=-\infty}^0e_{-\lambda+mz}(\nm_{Z/X})}{\prod_{m=-\infty}^ke_{-\lambda+mz}(\nm_{Z/X})}\frac{\prod_{m=-\infty}^0\lambda+mz}{\prod_{m=-\infty}^{-k}\lambda+mz}e^{-kz\partial_\lambda}\ff_Z,\\
	&(\sft^k\ff)_{\td X}=\frac{\prod_{m=-\infty}^0-[D]+\lambda+mz}{\prod_{m=-\infty}^{-k}-[D]+\lambda+mz}e^{-kz\partial_\lambda}\ff_{\td X}, 
	\end{split}
\end{equation} where $\ff_F$ is the restriction of the $H_{\wht T}^*(W)_\textrm{loc}$-valued function $\ff$ to the fixed component $F$.

\subsection{Fourier transforms}

We begin by introducing Iritani's \emph{continuous Fourier transform.} The key geometric input is the following theorem, which interprets the virtual localization formula in terms of Givental cones. Similar results can be found in \cite{Bro14} and \cite{FL20}; both authors attribute the idea to Givental. \begin{theorem}[\cite{Bro14,FL20}]\label{thm:givloc}
	Let $F$ be a $T$-fixed component of $W$. Let $\ff $ be a point in the $T$-equivariant Givental cone $\gcn_W$. Let $\ff_F$ be the restriction of $\ff$ to $F$. Then the Laurent expansion at $z=0$ of $\ff_F$ lies in the Givental cone of $F$ twisted by the normal bundle $\nm_{F/W}$ and the inverse equivariant Euler class $e_T^{-1}$.
\end{theorem}
\begin{proof}
	The restriction $\ff_F$ is of the form \[\ff_F=z+\bt_F+\sum_i\sum_{n,\delta}\phi_F^i\left\langle\bt(-\psi),\dots,\bt(-\psi),\frac{i_{F*}\phi_{F,i}}{z-\psi}\right\rangle_{0,n+1,\delta}^{W,T}\frac{\sq^\delta}{n!}\] where $\bt_F(z)=\bt(z)|_F$, $\{\phi_{F,i}\}$ is a basis of $H^*(F)$, $\{\phi_{F}^i\}$ is the dual basis so $(\phi_{F,i},\phi_{F}^j)_F=\delta_i^j$, and $i_F:F\to W$ is the inclusion map. By Theorem \ref{thm:loc}, contributions come from the following two types of $T$-invariant curves: \begin{itemize}
		\item Type-A: the $\psi$-class at the last marked point carries nontrivial $T$-weight; this happens when the last marked point lies in a component covering the base $\cp^1$.
		\item Type-B: the $\psi$-class at the last marked point carries trivial $T$-weight; this happens when the last marked point lies in a component contained in $F$.
	\end{itemize}
	We single out the type-A contributions and set \[\tau_F(z):=\bt_F(z)+\sum_i\sum_{n,\delta}\phi_F^i\left\langle\bt(-\psi),\dots,\bt(-\psi),\frac{i_{F*}\phi_{F,i}}{z-\psi}\right\rangle_{0,n+1,\delta}^{\textrm{type-A}}\frac{\sq^\delta}{n!}.\] Since $\psi$ is invertible on type-A fixed loci, we may instead write the denominator as $\frac{z}{\psi}-1$ and take the expansion of $\tau_F(z)$ at $z=0$, which gives an element in $\gsp^\textrm{tw}_+\subset\gsp^\textrm{tw}_F.$ We claim that the type-B contributions can be expressed as the $(\nm_{F/W},e_T^{-1})$-twisted Gromov--Witten invariants of $F$ with insertions $\tau_F(-\psi)$. More explicitly, we have \[\ff_F=z+\tau_F(z)+\sum_i\sum_{\substack{n\ge0\\d\in H_2(F,\Z)}}e_T(\nm_{F/W})\phi_F^i\left\langle\tau_F(-\psi),\dots,\tau_F(-\psi),\frac{\phi_{F,i}}{z-\psi}\right\rangle_{0,n+1,d}^{F,(\nm_{F/W},e_T^{-1})}\frac{\sq^{i_{F*}d}}{n!}.\] Then we see that the expansion at $z=0$ of $\ff_F$ lies in the twisted Givental cone of $F$. 
	
	It remains to prove the claim. In a type-B curve $\Sigma$, the components $C$ contained in $F$, viewed as a standalone curve, contributes to a Gromov--Witten invariant of $F$ twisted by $(\nm_{F/W},e_T^{-1})$. The other marked points of $C$ are either marked points of $\Sigma$ that happened to lie in $C$, or are originally nodes of $\Sigma$ joining $C$ to a component covering the base $\cp^1$. The first group of marked points have insertions of $\bt_F(-\psi)$. The second group of marked points have insertions of $\phi_F^i\left\langle\bt(-\psi),\dots,\bt(-\psi),\frac{i_{F*}\phi_{F,i}}{-\psi_k-\psi}\right\rangle$, where the factor $-\psi_k-\psi$ in the denominator is the virtual normal bundle contribution corresponding to smoothing the node. In total, the insertion is $\tau_F(-\psi)$. The factor $e_T(N_{F/W})$ comes from changing between the pairing on $W$ and the pairing on $F$. This proves the claim.
\end{proof}
Consider the weight decomposition of the normal bundle $\nm_{F/W}=\bigoplus_\alpha \nm_\alpha$. Let $\Delta_\alpha$ denote the quantum Riemann--Roch operator associated with the $(\nm_\alpha,e_\alpha^{-1})$-twist. Since $\ff_F$ lies in the $e_T^{-1}$-twisted Givental cone, in view of Theorem \ref{thm:qrr}, $\left(\prod_\alpha\Delta_\alpha^{-1}\right)\ff_F$ should lie in the untwisted Givental cone of $F$. Iritani writes down an operator $G_F$ with the property that its asymptotic expansion for $z\to0$ gives rise to $\prod_\alpha\Delta_\alpha$. Concretely, \[G_F=\prod_{\varrho}\frac{1}{\sqrt{-2\pi z}}(-z)^{-\varrho/z}\Gamma\left(-\frac{\varrho}{z}\right)\] where the product is taken over the $T$-equivariant Chern roots $\varrho$ of $\nm_{F/W}$.
Then Iritani computes the $z\to0$ asymptotics of an oscillatory integral involving $G_F$ (a Mellin-Barnes integral), using the stationary phase method. The result is that \[\int e^{\lambda\log S_F/z}G_F\ff_Fd\lambda\sim\sqrt{2\pi z}e^{c_F\lambda_j/z}\cft_{F,j}(\ff)\qquad \textrm{ as } z\to0\] for $j=0,1,\dots,|c_F|-1$, where
\[\cft_{F,j}(\ff):=\sqrt{\frac{\lambda_j}{c_F}}\left(\prod_\alpha(w_\alpha \lambda_j)^{-\rho_\alpha/z-r_\alpha/2}\right)\left(e^{z(\partial_u)^2/2}\Phi_\ff(u)\right)_{u=0}\] and \begin{equation*}\label{eq:bigphi}
	\Phi_\ff(u):=e^{-\frac{g(u,\lambda_j)}{z}+\frac{u}{\sqrt{c_F\lambda_j}}\left(1-\rho_F/z-r_F/2\right)}\left(\prod_\alpha\td\Delta_\alpha^{-1}\right)\ff_F.
\end{equation*} The relevant notations are listed below: \begin{itemize}
	\item $w_\alpha\in H_T^2(\pt,\Z)=\Z$ corresponds to $\alpha\in\Hom(T,S^1)$,
	\item $r_\alpha=\rank \nm_\alpha$, $r_F=\sum_\alpha r_\alpha$,
	\item $\rho_\alpha=c_1(\nm_\alpha)$, $\rho_F=\sum_\alpha \rho_\alpha$,
	\item $c_F=\sum_\alpha r_\alpha w_\alpha$,
	\item $S_F=\wht S^{\sigma_F(1)}$,
	\item $\lambda_j=e^{2\sqrt{-1}\pi j/c_F}S_F^{1/c_F}\prod_\alpha w_\alpha^{-r_\alpha w_\alpha/c_F}$,
	\item $g(u,\lambda_j)=\sum_{n=3}^\infty\frac{n-1}{n}\frac{u^n}{(c_F\lambda_j)^{n/2-1}}$,
	\item $\td\Delta_\alpha$ is the quantum Riemann--Roch operator associated with the $(\nm_\alpha,\td e_\alpha^{-1})$-twist.
\end{itemize} 
Extending linearly over $\nov\sq$, we obtain maps \[\cft_{F,j}:\gsp_W^\textrm{rat}\to S_F^{-{\rho_F}/{(c_Fz)}-{(r_F-1)}/{(2c_F)}}H^*(F)[z,z^{-1}]\laur{S_F^{-1/c_F}}\fps{\sq}.\] We also write \begin{align*}
	&q_{F,j}:=\sqrt{\lambda_j/c_F}\prod_\alpha(w_\alpha\lambda_j)^{-r_\alpha/2},\\
	&h_{F,j}:=-\frac{2\sqrt{-1}\pi j}{c_F}\rho_F+\sum_\alpha\left(\frac{r_\alpha w_\alpha}{c_F}\rho_F-\rho_\alpha\right)\log w_\alpha.
\end{align*}
For $\ff_F=\lambda^n\phi+O(\lambda^{n-1})$, the leading order term of $\cft_{F,j}(\ff)$ as a Laurent series of $S_F^{-1/c_F}$ is given by \begin{equation}\label{eq:cftlot}
	\cft_{F,j}(\ff)=q_{F,j}e^{h_{F,j}/z}S_F^{-\rho_F/(c_Fz)}\lambda_j^n\left(\phi+O(S_F^{-1/c_F})\right).
\end{equation}

\begin{proposition}[{\cite[Proposition 4.7]{Iri25}}]\label{prop:cftprop}
	The maps $\cft_{F,j}$ satisfy the properties: \begin{enumerate}[(1)]
		\item $\cft_{F,j}(\sftht^\beta\ff)=\wht S^\beta\cft_{F,j}(\ff)$,
		\item $\cft_{F,j}(\lambda\ff)=(zS\frac{\partial}{\partial S}+\lambda_j)\cft_{F,j}(\ff),$
		\item $\cft_{F,j}((z\partial_z-z^{-1}c_1^T(W)+\mu_W+\frac{1}{2})\ff)=(z\partial_z-z^{-1}(c_1(F)+c_F\lambda_j)+\mu_F)\cft_{F,j}(\ff).$
	\end{enumerate}
\end{proposition}
\begin{proof}
	From the definitions of $
	\sftht^\beta$ and $G_F$, we see that \[G_F(\sftht^\beta\ff)_F=Q^{\beta+\sigma_F(-\overline\beta)}e^{-z\overline\beta\partial_\lambda}(G_F\ff_F).\] Then \[\begin{split}
		\int e^{\lambda\log S_F/z}G_F(\sftht^\beta\ff)_Fd\lambda&=Q^{\beta+\sigma_F(-\overline\beta)}\int e^{\lambda\log S_F/z}e^{-z\overline\beta\partial_\lambda}G_F\ff_Fd\lambda\\
		&=Q^{\beta+\sigma_F(-\overline\beta)}\int e^{(\lambda+\overline\beta z)\log S_F/z}G_F\ff_Fd\lambda\\
		&=\wht S^\beta\int e^{\lambda\log S_F/z}G_F\ff_Fd\lambda.
	\end{split}\] Since $\cft_{F,j}$ is given by the asymptotics of the integral, this proves item (1). Item (2) follows from \[\int e^{\lambda\log S_F/z}G_F(\lambda\ff)_Fd\lambda=zS\partial_S\int e^{\lambda\log S_F/z}G_F\ff_Fd\lambda\] and the fact that \[zS\partial_S(e^{c_F\lambda_j/z})=\lambda_je^{c_F\lambda_j/z}.\] Item (3) follows from a degree argument; see \cite[Proposition 4.7]{Iri25} for details.
\end{proof}

The Fourier transform $\cft_{F,j}$ maps the equivariant Givental cone of $W$ to the Givental cone of $F$. More precisely:
\begin{proposition}[{\cite[Corollary 4.9]{Iri25}}]\label{prop:cft}
	The continuous Fourier transform of the equivariant J-function $J_W(\theta)$ of $W$ is of the form \[\cft_{F,j}(J_W(\theta))=S_F^{-\rho_F/(c_Fz)}M_F(\tau;Q_FS_F^{-\rho_F/c_F})v|_{Q_F\to\sq}\] for some $\tau=\tau(\bth)\in H^*(F)[S_F^{1/c_F},S_F^{-1/c_F}]\fps{\sq,\bth}$ and $v=v(\bth)\in q_{F,j}H^*(F)[z]\laur{S_F^{-1/c_F}}\fps{\sq}$, whose leading order terms are $h_{F,j}$ and $q_{F,j}$ respectively; $\tau$ is homogeneous of degree $2$ and $v/q_{F,j}$ is homogeneous of degree $0$; the subscript $Q_F\to\sq$ means to replace each $Q_F^d$ with $\sq^{i_{F*}d}$; and $\bth_FS_F^{\bullet/c_F}$ denotes the infinite set of variables $\{\theta_F^{i,k}S_F^{k/c_F}\}_{i,k}$.
\end{proposition}
\begin{proof}
	Set $\ff=zJ_W(\theta)$, and consider its restriction $\ff_F$. By Theorem \ref{thm:givloc}, $\ff_F$ lies in the $(\nm_{F/W},e_T^{-1})$-twisted Givental cone of $F$. Hence it is of the form $\bigj_F^\textrm{tw}(\bt)$, where $\bt(z)$ is the non-negative part of the Laurent expansion of $\ff_F-z$ at $z=0$, which we denote by $[\ff_F-z]_+$. In other words, we have \[\cft_{F,j}(zJ_W)=\wht\bigj_{F,j}^\textrm{tw}(\bt)|_{Q_F\to\sq,\bt=[\ff_F-z]_+}\] where $\wht\bigj_{F,j}^\textrm{tw}$ is the Fourier transform of the big J-function $\bigj_{F}^\textrm{tw}$. The problem can then be reduced to showing that \[\wht\bigj_{F,j}^\textrm{tw}(\bt)=S_F^{-\rho_F/(c_Fz)}zM_F(\tau;Q_FS_F^{-\rho_F/c_F})v\] for some $\tau=\tau(\bt)$ and $v=v(\bt)$ with the specified leading order terms; see \cite[Corollary 4.9]{Iri25} for details.
	
	Consider the big J-function $\td{\bigj}_F^\textrm{tw}$ of the $(\nm_{F/W},\td{e}_T^{-1})$-twisted theory. Note that \[\td{\bigj}_F^\textrm{tw}(\bt)=\bigj_F^\textrm{tw}(\bt)|_{Q_F\to Q_F\prod_\alpha (w_\alpha\lambda)^{\rho_\alpha}}.\] By Theorem \ref{thm:qrr}, $(\prod_\alpha\td\Delta_\alpha^{-1})\td{\bigj}_F^\textrm{tw}$ lies in the Givental cone $\gcn_F$. We claim that $\gcn_F$ is invariant under $\ff(Q_F)\mapsto e^{sh/z}\ff(Q_Fe^{sh})$ for $h\in H^2(F)$ and a formal parameter $s$. To see this, consider the shift of coordinates $\tau\mapsto\tau+sh$ for the fundamental solution $M_F(\tau)$. The derivative of $M_F(\tau+sh)\phi_i$ with respect to $s$ is \[\frac{\partial}{\partial s}M_F(\tau+sh)\phi_i=\sum_j\sum_{(n,d)\ne(0,0)}\phi^j\left\langle\phi_i, h,\tau+sh,\dots,\tau+sh,\frac{\phi_j}{z-\psi}\right\rangle_{0,n+2,d}\frac{Q^d}{(n-1)!}.\] By the divisor equation, for $(n,d)\ne(0,0),(1,0)$ we have \[\begin{split}
		\left\langle \phi_i, h,\tau+sh,\dots,\tau+sh,\frac{\phi_j}{z-\psi}\right\rangle_{0,n+2,d}=
		\left\langle\phi_i,\tau+sh,\dots,\tau+sh,\frac{h\phi_j}{z(z-\psi)}\right\rangle_{0,n+1,d}\\
		+(h\cdot d)\left\langle\phi_i,\tau+sh,\dots,\tau+sh,\frac{\phi_j}{z-\psi}\right\rangle_{0,n+1,d}.
	\end{split}\] We also have \[\sum_j\phi^j\left\langle\phi_i, h,\frac{\phi_j}{z-\psi}\right\rangle_{0,3,0}=\frac{h}{z}\phi_i.\] We conclude that \[\frac{\partial}{\partial s}M_F(\tau+sh)=\left(\frac{h}{z}+hQ\partial_Q\right)M_F(\tau+sh).\] Integrating over $s$, we obtain \begin{equation}\label{eq:divflow}
		M_F(\tau+sh;Q)=e^{sh/z}M_F(\tau;Qe^{sh}).
	\end{equation} Since $\gcn_F$ is the union of $zM_F(\tau)\gsp_+$ for all $\tau\in H^*(F)\fps{Q}$, this proves the claim. Similarly, $\gcn_F$ is invariant under $\ff\mapsto e^{s/z}\ff$. To show this we use the string equation; the proof is otherwise identical and left to the reader. Finally, by \cite[Lemma 2.7]{IK23}, $\gcn_F$ is invariant under the operator $e^{z(\partial_u)^2/2}=e^{(z\partial_u)^2/(2z)}$. (We remark that all the invariance properties of the Givental cone mentioned above are essentially contained in \cite{CG07}; see also \cite{CCIT09}.)
	
	We set $s=u/\sqrt{c_F\lambda_j}$. Then we deduce that \[e^{z(\partial_u)^2/2}|_{u=0}e^{-g(u,\lambda_j)/z}e^s\left(\prod_\alpha e^{-s(\rho_\alpha/z+r_\alpha/2)}\td\Delta_\alpha^{-1}\right)\td{\bigj}_F^\textrm{tw}(\bt)\Bigr|_{\substack{Q_F\to Q_Fe^{-s\rho_F}\\ \lambda=\lambda_je^s\phantom{aaaaaa}}}\] lies in $\gcn_F$. Thus it equals $zM_F(\td\tau)\td v$ for some $\td\tau\in H^*(F)\fps{\lambda_j^{-1},Q_F,\bt\lambda_j^\bullet}$ with leading order term 0 and $\td v\in H^*(F)[z]\fps{\lambda_j^{-1},Q_F,\bt\lambda_j^\bullet}$ with leading order term 1. Hence \[\wht\bigj_{F,j}^\textrm{tw}(\bt)=q_{F,j}\left(\prod_\alpha(w_\alpha\lambda_j)^{-\rho_\alpha/z}\right)zM_F(\td\tau)\td v\Bigr|_{Q_F\to Q_F\prod_\alpha(w_\alpha\lambda_j)^{-\rho_\alpha}}.\] We set $\tau=h_{F,j}+\td\tau(Q_F\prod_\alpha(w_\alpha\lambda_j)^{-\rho_\alpha})$ and $v=q_{F,j}\td v(Q_F\prod_\alpha(w_\alpha\lambda_j)^{-\rho_\alpha})$. Then applying \eqref{eq:divflow} again, we obtain \[\wht\bigj_{F,j}^\textrm{tw}(\bt)=S_F^{-\rho_F/(c_Fz)}zM_F(\tau;Q_FS_F^{-\rho_F/c_F})v,\] which concludes the proof.
\end{proof}

Let $Y=X$ or $\td X$ (so we exclude the $F=Z$ case). In these two cases, Iritani introduces a simpler map called the \emph{discrete Fourier transform}: \begin{equation*}
	\dft_Y(\ff):=\sum_{k\in\Z}S^k\kappa_Y(\sft^{-k}\ff)
\end{equation*} where $\kappa_Y$ is the Kirwan map. Define the extended Givental spaces \[\gsp_Y^\textnormal{ext}:=H^*(Y)[z,z^{-1}]\fps{C_Y}.\] Note that the Givental space $\gsp_Y$ of $Y$ embeds in $\gsp_Y^\textnormal{ext}$ via the dual Kirwan map $\kappa_Y^*$.

\begin{proposition}[{\cite[Proposition 4.1]{Iri25}}]\label{prop:dftdef}
	If $\ff$ lies in a tangent space of the equivariant Givental cone $\gcn_W$, then $\dft_Y(\ff)$ is a well-defined element in $\gsp_Y^\textnormal{ext}$.
\end{proposition}
\begin{proof}
	Recall that a tangent space of $\gcn_W$ is of the form $M_W\gsp_+$, and is invariant under shift operators. Hence $\sft^{-k}\ff$ is also of the form $M_Wv$, and upon examining \eqref{eq:fund}, we see that $\kappa_Y(\sft^{-k}\ff)$ is well-defined. Thus it remains to check that $\dft_Y(\ff)$ lies in $\gsp_Y^\textnormal{ext}$. We only explain the $Y=X$ case since the $Y=\td X$ case is spelled out in \cite[Proposition 4.1]{Iri25}. We must verify that $\kappa_X(\sft^k\ff)$ lies in $x^kH^*(X)[z,z^{-1}]\fps\sq$ if $k>0$ and lies in $y^kH^*(X)[z,z^{-1}]\fps\sq$ if $k<0$. The first half is obvious upon inspecting the first equation in \eqref{eq:shft}.
	
	It remains verify the $k<0$ case. Consider the expansion $\ff=\sum_{\delta\in\eff(W)}\ff_\delta\sq^\delta.$ For $\ff_\delta$ to contribute to $\kappa_X(\sft^k\ff)$, it must satisfy the property that $e^{-kz\partial_\lambda}\ff_{\delta,X}$ has a factor of $\lambda$ in the denominator; otherwise, a factor of $\lambda$ in the numerator of \[\frac{\prod_{m=-\infty}^0-\lambda+mz}{\prod_{m=-\infty}^k-\lambda+mz}\] would survive, and then the Kirwan map $\kappa_X$ sets $\lambda=0$. Hence $\ff_{\delta,X}$ must have a factor of $\lambda+kz$ in the denominator, in order to contribute. Since $\ff$ is of the form $M_Xv$, this factor must come from a correlator of the form \[\left\langle\phi_i,\theta,\dots,\theta,\frac{\phi_j}{z-\psi}\right\rangle_{0,n+2,\delta'}\] for some $\delta'$ with $\delta\in\delta'+\eff(W)$. By Theorem \ref{thm:loc}, this must come from a $T$-invariant curve whose last marked point maps to $X$ and lies in a component which is a $(-k)$-fold cover of the base $\cp^1$, and thus has class $-k$ times $(0,1,0,0)$ or $(0,1,-1,0)$ (cf. proof of Lemma \ref{lem:wclass}). Hence $\delta'\in(0,-k,k,0)+\eff(W)\subseteq(0,0,k,0)+\eff(W).$
\end{proof}
  
\begin{proposition}[{\cite[Section 4.1]{Iri25}}]\label{prop:dftshft}
	The map $\dft_Y$, when defined, satisfies \begin{enumerate}[(1)]
	\item $\dft_Y(\sft^l\ff)=S^l\dft_Y(\ff)$,
	\item $\dft_Y((z\xi\sq\partial_\sq+\xi)\ff)=(z\xi\wht S\partial_{\wht S}+\kappa_Y(\xi))\dft_Y(\ff)$,
	\item $\dft_Y((z\partial_z-z^{-1}c_1^T(W)+\mu_W+\frac{1}{2})\ff)=(z\partial_z-z^{-1}c_1(Y)+\mu_Y)\dft_Y(\ff).$
\end{enumerate} Here we use notation from the definition of the (equivariant) quantum connection.
\end{proposition}
The discrete Fourier transform $\dft_Y$ also maps the equivariant Givental cone of $W$ to the Givental cone of  $Y$. More precisely:
\begin{proposition}[{\cite[Corollary 4.11]{Iri25}}]\label{prop:dft}
	Let $Y=X$ or $\td X$. Then \[\dft_Y(zJ_W(\theta))=zM_Y(\tau)v\] for some $\tau\in H^*(Y)\fps{C_Y,\bth}$ and $v\in H^*(Y)[z]\fps{C_Y,\bth}$, such that $\tau|_{\bth=0}\equiv S_Y^{1/c_Y}$ and $v|_{\bth=0}\equiv 1$ modulo the closed ideal of $\C[z]\fps{C_Y}$ generated by $\sq^\delta$ for $\delta\in\eff(W)\setminus\{0\}$ and $(yS^{-1})^{c_Y}$.
\end{proposition}
\begin{proof}
	By Proposition \ref{prop:cft}, there exist $\tau\in H^*(Y)[S_Y^{1/c_Y},S_Y^{-1/c_Y}]\fps{\sq,\bth}$ with leading order term $h_{Y,0}=0$ and $v\in H^*(Y)[z]\laur{S_Y^{1/c_Y}}\fps{\sq,\bth}$ with leading order term $q_{Y,0}=c_Y^{-1}$ such that \[S^{\rho_Y/(c_Yz)}\cft_{Y,0}(J_W(\theta))=M_Y(\tau;Q_YS_Y^{-\rho_Y/c_Y})v|_{Q_Y\to\sq}.\] Note that $M_Y(\tau;Q_YS_Y^{-\rho_Y/c_Y})$ is precisely the image of $M_Y(\tau)$ induced by the dual Kirwan map $\kappa_Y^*:\nov{Q_Y}\to\nov{C_Y}$, so we henceforth write it as $M_Y(\tau)$. After an explicit computation (see \cite[Proposition 4.10]{Iri25} for details), we find that $\dft_Y$ and $\cft_{Y,0}$ are related by the following simple formula: \[\exp({S_Y^{1/c_Y}/z})\cft_{Y,0}(\ff)=c_Y^{-1}S_Y^{-\rho_Y/(c_Yz)}\dft_Y(\ff).\] Then \[\dft_{Y}(J_W)=\exp(S_Y^{1/c_Y}/z)M_Y(\tau)c_Yv=M_Y(\tau+S_Y^{1/c_Y})c_Yv,\] and the conclusion follows.
\end{proof}

\subsection{Quantum D-module isomorphisms}

Let $\qdm_T(W)$ be the $T$-equivariant quantum D-module of $W$. Let $\qdm_T(W)[\sq^{-1}]$ be its localization with respect to the multiplicative set $\{\sq^\delta\}_{\delta\in\eff(W)}.$ It has the structure of a $\C[H_2^T(W,\Z)]$-module via the action of shift operators. We consider the subring $\C[C_{\td X}]\subset\C[H_2^T(W,\Z)]$, and the $\C[C_{\td X}]$-submodule of $\qdm_T(W)[\sq^{-1}]$ generated by $\qdm_T(W)$: \[\qdm_T(W)_{\td X}:=\C[C_{\td X}]\cdot\qdm_T(W)\subset\qdm_T(W)[\sq^{-1}].\] By Proposition \ref{prop:shftact}, \[\qdm_T(W)_{\td X}=\C[yS^{-1}]\cdot\qdm_T(W),\] so it is a module over $\C[z]\fps{\sq,\bth}[S,yS^{-1}]$. Let $\idl$ be the ideal generated by $S$ and $yS^{-1}$. We consider the graded $\idl$-adic completion \[\qdm_T(W)_{\td X}\sphat:=\varprojlim_k\qdm_T(W)_{\td X}/\idl^k\qdm_T(W)_{\td X}.\] It is a module over $\C[z]\fps{C_{\td X},\bth}$. By \eqref{eq:shift}, the action of $\lambda$ on $\qdm_T(W)$ induces an action of $\lambda$ on $\qdm_T(W)_{\td X}\sphat$; by Corollary \ref{cor:shftcomm}, the quantum connection on $\qdm_T(W)$ induces operators $z\nabla_{\theta^{i,k}}$, $z\nabla_{z\partial_z}$, $z\nabla_{\xi\sq\partial_\sq}$ on $\qdm_T(W)_{\td X}\sphat$. 

Next, we introduce a suitable extension of the Novikov ring for $\qdm(\td X)$. Recall the extension $\nov{\td Q}\hookrightarrow\nov{C_{\td X}}$ given by the dual Kirwan map. We define \[\qdm(\td X)^\textrm{ext}:=\qdm(\td X)\otimes_{\C[z]\fps{\td Q,\td\tau}}\C[z]\fps{C_{\td X},\td\tau}.\]
The quantum connection on $\qdm(\td X)$ can be naturally extended to $\qdm(\td X)^\textrm{ext}$: the operators $\nabla_{\td\tau^i},\nabla_{z\partial_z},\nabla_{\xi\wht S\partial_{\wht S}}:\qdm(\td X)^\textrm{ext}\to z^{-1}\qdm(\td X)^\textrm{ext}$ (where $\xi\in H^2(\td X)$) are given by \begin{align*}
	\nabla_{\td\tau^i}&=\nabla_{\td\tau^i}\otimes\id+\id\otimes\partial_{\td\tau^i},\\
	\nabla_{z\partial_z}&=\nabla_{z\partial_z}\otimes\id+\id\otimes z\partial_z,\\
	\nabla_{\xi\wht S\partial_{\wht S}}&=\nabla_{\kappa_{\td X}(\xi)\td Q\partial_{\td Q}}\otimes\id+\id\otimes \xi\wht S\partial_{\wht S}.
\end{align*}
By Proposition \ref{prop:dft}, \[\dft_{\td X}(J_W(\theta))=M_{\td X}(\td \tau(\theta))\td v(\theta)\] for $\td \tau(\theta)\in H^*(X)\fps{C_{\td X},\bth}$ and $\td v(\theta)\in H^*(X)[z]\fps{C_{\td X},\bth}$ such that $\td\tau(0)\equiv S$ and $\td v(0)\equiv1$ modulo the closed ideal generated by $\sq$ and $yS^{-1}$. We consider the pull-back \[\td\tau^*\qdm(\td X)^\textrm{ext}:=H^*(\td X)[z]\fps{C_{\td X},\bth}\] equipped with the pull-back connection \begin{align*}
	\nabla_{\theta^{i,k}}&=\partial_{\theta^{i,k}}+z^{-1}(\partial_{\theta^{i,k}}\td\tau(\theta))\star_{\td\tau(\theta)},\\
	\nabla_{z\partial_z}&=z\partial_z-z^{-1}{E_{\td X}\star_{\td\tau(\theta)}}+\mu_{\td X},\\
	\nabla_{\xi\wht S\partial_{\wht S}}&=\xi\wht S\partial_{\wht S}+z^{-1}{\kappa_{\td X}(\xi)\star_{\td\tau(\theta)}}+z^{-1}(\xi\wht S\partial_{\wht S}\td\tau(\theta))\star_{\td\tau(\theta)}.
\end{align*}  

\begin{theorem}[{\cite[Theorem 5.2]{Iri25}}]\label{thm:qdmiso}
	There is an isomorphism \[\ft_{\td X}\sphat:\qdm_T(W)_{\td X}\sphat\to\td\tau^*\qdm(\td X)^\textnormal{ext}\] of $\C[z]\fps{C_{\td X},\bth}$-modules which \begin{enumerate}[(1)]
	\item intertwines $\seiht^\beta(\theta)$ with $\wht S^\beta$ for $\beta\in H_2^T(W,\Z)$,
	\item intertwines $z\nabla_{\xi\sq\partial_\sq}$ with $z\nabla_{\xi\wht S\partial_{\wht S}}$ for $\xi\in H_T^2(W)$,
	\item commutes with $\nabla_{\theta^{i,k}}$,
	\item intertwines $\nabla_{z\partial_z}+\frac{1}{2}$ with $\nabla_{z\partial_z}$,
	\item is homogeneous of degree $0$.
\end{enumerate} If $s_1,\dots,s_N\in H_T^*(W)$ are homogeneous elements such that $\kappa_{\td X}(s_1),\dots,\kappa_{\td X}(s_N)$ form a basis of $H^*(\td X)$, then $s_1,\dots,s_N$ form a basis $\qdm_T(W)_{\td X}\sphat$ over $\C[z]\fps{C_{\td X},\bth}$.
\end{theorem}
\begin{proof}
	Differentiating the equation \[\dft_{\td X}(J_W(\theta))=M_{\td X}(\td\tau(\theta))\td v(\theta)\] by $\theta^{i,k}$, we obtain by \eqref{eq:fundsol} that \[\dft_{\td X}(M_W(\theta)z^{-1}\phi_i\lambda^k)=M_{\td X}(\td\tau(\theta))\nabla_{\theta^{i,k}}\td v(\theta)\] where $\{\phi_i\lambda^k\}$ is the $\C$-basis of $H_T^*(W)$ dual to $\{\theta^{i,k}\}$. We define $\ft_{\td X}$ to be the $\C[z]\fps{\sq,\bth}$-module map $\qdm_T(W)\to\td\tau^*\qdm({\td X})^\textrm{ext}$ sending $\phi_i\lambda^k$ to $z\nabla_{\theta^{i,k}}\td v(\theta)$, and extend it to the localization by $\sq$. Thus we have the commutative diagram \begin{equation}\label{eq:diagm}
		\begin{tikzcd}
	{\qdm_T(W)[\sq^{-1}]} & {\td\tau^*\qdm(\td X)^\textnormal{ext}[\sq^{-1}]} \\
	{\gsp^\textnormal{rat}_W[\sq^{-1}]} & {\gsp_{\td X}^\textnormal{ext}[\sq^{-1}]}
	\arrow["{\ft_{\td X}}", from=1-1, to=1-2]
	\arrow["{M_W(\theta)}"', from=1-1, to=2-1]
	\arrow["{M_{\td X}(\td\tau(\theta))}", from=1-2, to=2-2]
	\arrow["{\dft_{\td X}}"', dashed, from=2-1, to=2-2]
\end{tikzcd}
	\end{equation} where a dashed arrow is used for $\dft_{\td X}$ as it is not defined throughout $\gsp^\textnormal{rat}_W[\sq^{-1}]$; however it is well-defined on the image of $M_W(\theta)$ by Proposition \ref{prop:dftdef}. The map $\ft_{\td X}$ satisfies properties (1)-(5) by Proposition \ref{prop:dftshft}, \eqref{eq:fundsol}, \eqref{eq:fsolshft}, and the fact that $\dft_{\td X}\circ\partial_{\theta^{i,k}}=\partial_{\theta^{i,k}}\circ\dft_{\td X}.$ By (1), \[\ft_{\td X}(\qdm_T(W)_{\td X})\subset\td\tau^*\qdm(\td X)^\textrm{ext},\] and $\ft_{\td X}$ extends to a map \[\ft_{\td X}\sphat:\qdm_T(W)_{\td X}\sphat\to\td\tau^*\qdm(\td X)^\textrm{ext}.\] This still satisfies (1)-(5).

It remains to prove that $\ft_{\td X}\sphat$ is an isomorphism. Write $R:=\C[z]\fps{\sq,\bth}[S,yS^{-1}],\mdl_0:=\qdm_T(W),\mdl:=\qdm_T(W)_{\td X}$. Let $\idl\subset R$ be the ideal generated by $S,yS^{-1}.$ Let $s_1,\dots,s_N\in H_T^*(W)$ be homogeneous elements such that $\kappa_{\td X}(s_1),\dots,\kappa_{\td X}(s_N)$ form a basis of $H^*(\td X)$. We will show that $\mdl/\idl\mdl$ is generated by $s_1,\dots,s_N$ as a module over $R/\idl=\C[z]\fps{\sq,\bth}/(y)$. Since $S$ preserves $\mdl_0$, we have \[\mdl/\idl\mdl=\frac{\mdl_0+yS^{-1}\C[yS^{-1}]\mdl_0}{S\mdl_0+yS^{-1}\C[yS^{-1}]\mdl_0}\cong\frac{\mdl_0}{S\mdl_0+(\mdl_0\cap yS^{-1}\C[yS^{-1}]\mdl_0)}.\] Let $\lt(f)$ denote the leading order term $f|_{\sq=\bth=0}$ of $f\in\mdl_0$. By Lemma \ref{lem:tech}, for any $f\in\mdl_0$, we can find $a_1,\dots,a_N\in\C[z]$ and $c_1,c_2,c_3\in H_T^*(W)[z]$ such that \[\lt(f)-\sum_{i=1}^Na_is_i=\lt(\sei c_1+x\sei^{-1}c_2+y\sei^{-1}c_3).\] If $f$ is homogeneous, then clearly $a_i,c_i$ can be chosen to be homogeneous. Applying the same argument term by term in the $(\sq,\bth)$-power series $f$, we can construct $A_i\in\C[z]\fps{\sq,\bth}$, $C_i\in H_T^*(W)[z]\fps{\sq,\bth}$ such that \[f-\sum_{i=1}^NA_is_i=\sei C_1+x\sei^{-1}C_2+y\sei^{-1}C_3.\] The right hand side lies in ${S\mdl_0+(\mdl_0\cap yS^{-1}\C[yS^{-1}]\mdl_0)}$, since $xS^{-1}\mdl_0=yS^{-1}(xy^{-1})\mdl_0\subset yS^{-1}\mdl_0$ by Proposition \ref{prop:shftact}. We conclude that $\mdl/\idl\mdl$ is generated by $s_1,\dots,s_N$. Hence the completion $\qdm_T(W)_{\td X}\sphat=\varprojlim_k\mdl/\idl^k\mdl$ is generated by $s_1,\dots,s_N$ as a module over $\C[z]\fps{C_{\td X},\bth}=\varprojlim_kR/\idl^k.$ Since $\td\tau|_{\sq=S=yS^{-1}=\bth=0}=0$, by diagram \eqref{eq:diagm} we have \[\ft_{\td X}(s_i)|_{\sq=S=yS^{-1}=\bth=0}=\dft_{\td X}(s_i)|_{\sq=S=yS^{-1}=\bth=0}=\kappa_{\td X}(s_i).\] This implies that $\ft_{\td X}(s_i)$ is a $\C[z]\fps{C_{\td X},\bth}$-basis of $\td\tau^*\qdm(\td X)^\textrm{ext}$. Consequently, $s_1,\dots,s_N\in\qdm_T(W)_{\td X}\sphat$ are linearly independent over $\C[z]\fps{C_{\td X},\bth}$, and $\ft_{\td X}\sphat$ is an isomorphism.
\end{proof}

\begin{lemma}[{\cite[Lemma 5.3]{Iri25}}]\label{lem:tech}
	Any element in $\ker(\kappa_{\td X})\subset H_T^*(W)[z]$ is of the form $\lt(\sei c_1+x\sei^{-1}c_2+y\sei^{-1}c_3)$ for some $c_1,c_2,c_3\in H_T^*(W)[z]$.
\end{lemma}
The lemma is obtained using \eqref{eq:fsolshft}, \eqref{eq:shft}, Theorem \ref{thm:loc}, Lemma \ref{lem:wclass}, \ref{lem:abloc}, \ref{lem:kerkir}.

We next turn to $X$. 
Using the extension $\nov Q\hookrightarrow\nov{C_X}$ given by the dual Kirwan map, we define \[\qdm(X)^\textrm{ext}:=\qdm(X)\otimes_{\C[z]\fps{Q,\tau}}\C[z]\fps{C_X,\tau}.\]
The quantum connection on $\qdm(X)$ extends to $\qdm(X)^\textrm{ext}$ in the same way as before.
By Proposition \ref{prop:dft}, \[\dft_X(J_W(\theta))=M_X(\tau(\theta))v(\theta)\] for $\tau(\theta)\in H^*(X)\fps{C_X,\bth}$ and $v(\theta)\in H^*(X)[z]\fps{C_X,\bth}$ such that $\tau(0)\equiv xS^{-1}$ and $v(0)\equiv1$ modulo the closed ideal generated by $\sq$ and $y^{-1}S$. We consider the pull-back \[\tau^*\qdm(X)^\textrm{ext}:=H^*(X)[z]\fps{C_X,\bth}\] equipped with the pull-back connection.
Denote $\q:=yS^{-1}$. Let $s=r-1$ if $r$ is even and $2(r-1)$ if $r$ is odd. We consider the extension \[\C[z]\fps{C_{X}}\subset \C[z]\laur{\q^{-1/s}}\fps\sq.\] Then as before, the quantum connection extends to \[\tau^*\qdm(X)^\textrm{La}:=H^*(X)[z]\laur{\q^{-1/s}}\fps{\sq,\bth}.\] Note that $\C[z]\laur{\q^{-1/s}}\fps\sq$ also contains $\C[z]\fps{C_{\td X}}$.
\begin{proposition}[{\cite[Proposition 5.4]{Iri25}}]
	There is a homomorphism \[\ft_X\sphat:\qdm_T(W)_{\td X}\sphat\to\tau^*\qdm(X)^\textnormal{La}\] of $\C[z]\fps{C_{\td X},\bth}$-modules satisfying the properties (1)-(5) of Theorem \ref{thm:qdmiso}.
\end{proposition}
\begin{proof}
	Define $\ft_X:\qdm_T(W)\to\tau^*\qdm(X)^\textrm{ext}$ in the same way as $\ft_{\td X}$. It satisfies properties (1)-(5). By property (1), this induces a map $\ft_X:\qdm_T(W)_{\td X}\to\tau^*\qdm(X)^\text{ext}[\q]$. It extends to a map $\ft_X\sphat:\qdm_T(W)_{\td X}\sphat\to\tau^*\qdm(X)^\text{ext}\laur{\q^{-1}}$ since $\q$ has positive degree; see \cite[Proposition 5.4]{Iri25} for details.
\end{proof}
Finally, we treat $Z$. Let $\sigma$ be the bulk variable for $Z$. We consider the extension \[\C[z]\fps{Q_Z,\sigma}\to\C[z]\laur{\q^{-1/s}}\fps{\sq,\sigma},\ Q_Z^d\mapsto Q^{\iota_*d}\q^{-\rho_Z\cdot d/(r-1)}=\sq^{i_{Z*}d}S_Z^{-\rho_Z\cdot d/c_Z}\] where $\iota:Z\to X,\ i_Z:Z\to W$ are the inclusions, and $\rho_Z=c_1(\nm_{Z/W})=c_1(\nm_{Z/X}).$ It is degree preserving, but not necessarily injective. We introduce \[\qdm(Z)^\textrm{La}:=\qdm(Z)\otimes_{\C[z]\fps{Q_Z,\sigma}}\C[z]\laur{\q^{-1/s}}\fps{\sq,\sigma}\] with the extended connection \[\begin{split}
	\nabla_{\sigma^i}&=\partial_{\sigma^i}+z^{-1}\phi_{Z,i}\star_\sigma,\\
	\nabla_{z\partial_z}&=z\partial_z-z^{-1}{E_Z\star_\sigma}+\mu_Z,\\
	\nabla_{\xi\wht S\partial_{\wht S}}&=\xi\wht S\partial_{\wht S}+z^{-1}{\kappa_Z(\xi)\star_\sigma},
\end{split}\] where $\kappa_Z:H_T^2(W)\to H^2(Z)$ sends $\xi$ to $i_Z^*\xi|_{\lambda=\rho_Z/(r-1)}$; this is not a Kirwan map.

Let $j\in\{0,\dots,r-2\}.$ We have \[\lambda_{j}=e^{-\frac{2\sqrt{-1}\pi}{r-1}\left(j+\frac{r}{2}\right)}\q^{\frac{1}{r-1}}.\] By Proposition \ref{prop:cft}, \[\q^{\rho_z/((r-1)z)}\cft_{Z,j}(J_W(\theta))=M_Z(\sigma_j(\theta))u_j(\theta)\] for some $\sigma_j(\theta)\in H^*(Z)\laur{\q^{-1/(r-1)}}\fps{\sq,\bth}$ with leading order term \[h_{Z,j}=\frac{2\sqrt{-1}\pi}{r-1}\left(j+\frac{1}{2}\right)\rho_Z\] and $u_j(\theta)\in q_{Z,j}H^*(Z)[z]\laur{\q^{-1/(r-1)}}\fps{\sq,\bth}$ with leading order term \[q_{Z,j}=\frac{1}{\sqrt{r-1}}e^{\frac{\sqrt{-1}\pi}{r-1}(jr+\frac{1}{2})}\q^{-\frac{r}{2(r-1)}}.\] We introduce \[\sigma_j^*\qdm(Z)^\textrm{La}:=H^*(Z)[z]\laur{\q^{-1/s}}\fps{\sq,\bth}\] equipped with the pull-back connection \[\begin{split}
	\nabla_{\theta^{i,k}}&=\partial_{\theta^{i,k}}+z^{-1}(\partial_{\theta^{i,k}}\sigma_j(\theta))\star_{\sigma_j(\theta)},\\
	\nabla_{z\partial_z}&=z\partial_z-z^{-1}{E_Z\star_{\sigma_j(\theta)}}+\mu_Z,\\
	\nabla_{\xi\wht S\partial_{\wht S}}&=\xi\wht S\partial_{\wht S}+z^{-1}{\kappa_Z(\xi)\star_{\sigma_j(\theta)}}+z^{-1}\xi\wht S\partial_{\wht S}{\sigma_j(\theta)\star_{\sigma_j(\theta)}}.
\end{split}\]
It is convenient to shift the map $\sigma_j(\theta)$ in the identity direction: \[\varsigma_j(\theta):=\sigma_j(\theta)-(r-1)\lambda_j,\] and introduce $\varsigma_j^*\qdm(Z)^\textrm{La}$ with the same underlying module $H^*(Z)[z]\laur{\q^{-1/s}}\fps{\sq,\bth}$ equipped with the connection \[\varsigma_j^*\nabla=\sigma_j^*\nabla+(r-1)\lambda_j\frac{dz}{z^2}-\frac{\lambda_j}{z}\frac{d\q}{\q}.\] Henceforth we shall write $\nabla$ for $\varsigma_j^*\nabla$.
\begin{proposition}[{\cite[Proposition 5.7]{Iri25}}]
	For $j\in\{0,\dots,r-2\}$, there is a homomorphism \[\ft_{Z,j}\sphat:\qdm_T(W)_{\td X}\sphat\to\varsigma_j^*\qdm(Z)^\textnormal{La}\] of $\C[z]\fps{C_{\td X},\bth}$-modules which \begin{enumerate}[(1)]
		\item intertwines $\seiht^\beta$ with $\sftht^\beta$ for $\beta\in H_2^T(W)$,
		\item intertwines $z\nabla_{\xi\sq\partial_\sq}$ with $z\nabla_{\xi\wht S\partial_{\wht S}}$ for $\xi\in H_T^2(W)$,
		\item commutes with $\nabla_{\theta^{i,k}}$,
		\item intertwines $\nabla_{z\partial_z}+\frac{1}{2}$ with $\nabla_{z\partial_z}$,
		\item is homogeneous of degree $-r$.
	\end{enumerate}
\end{proposition}
\begin{proof}
	Define $\ft_{Z,j}:\qdm_T(W)\to\sigma_j^*\qdm(Z)^\textrm{La}$ by sending $\phi_i\lambda^k$ to $\nabla_{\theta^{i,k}}u_j(\theta)$, and extend it to the localization over $\sq$. It fits in the diagram \begin{equation}\label{eq:zdiagm}
		\begin{tikzcd}
	{\qdm_T(W)[\sq^{-1}]} & {\sigma_j^*\qdm(Z)^\textnormal{La}[\sq^{-1}]} \\
	{\gsp^\textnormal{rat}_W[\sq^{-1}]} & {\gsp_{Z}^\textnormal{La}[\sq^{-1}]}
	\arrow["{\ft_{Z,j}}", from=1-1, to=1-2]
	\arrow["{M_W(\theta)}"', from=1-1, to=2-1]
	\arrow["{M_{Z}(\sigma_j(\theta))}", from=1-2, to=2-2]
	\arrow["{\q^{\rho_Z/((r-1)z)}\cft_{Z,j}}"', from=2-1, to=2-2]
\end{tikzcd}
	\end{equation} where $\gsp_{Z}^\textnormal{La}:=H^*(Z)[z,z^{-1}]\laur{\q^{-1/s}}\fps{\sq}$. Properties (1)-(5) follow from Proposition \ref{prop:cftprop} as with $\ft_{\td X}$, and after accounting for the shift from $\sigma_j^*\nabla$ to $\varsigma_j^*\nabla$. The map extends to $\qdm_T(W)_{\td X}\sphat$ for the same reason as $\ft_X$.
\end{proof}
\begin{theorem}[{\cite[Theorem 5.9]{Iri25}}]
	The map \[\Phi:=\ft_X\sphat\oplus\bigoplus_{j=0}^{r-2}\ft_{Z,j}\sphat:\qdm_T(W)_{\td X}\sphat\to\tau^*\qdm(X)^\textnormal{La}\oplus\bigoplus_{j=0}^{r-2}\varsigma_j^*\qdm(Z)^\textnormal{La}\] induces an isomorphism after base change to $\C[z]\laur{\q^{-1/s}}\fps{\sq,\bth}$. Consequently, the map \[\Psi:=\Phi\circ(\ft_{\td X}\sphat)^{-1}:\td\tau^*\qdm(\td X)^\textnormal{ext}\to\tau^*\qdm(X)^\textnormal{La}\oplus\bigoplus_{j=0}^{r-2}\varsigma_j^*\qdm(Z)^\textnormal{La}\] induces an isomorphism after base change to $\C[z]\laur{\q^{-1/s}}\fps{\sq,\bth}$. Moreover, $\Psi$ intertwines the quantum connections, and intertwines the pairing $P_{\td X}$ with $P_X\oplus\bigoplus_{j=0}^{r-2}P_Z.$
\end{theorem}
\begin{proof}
	Let $\{\phi_{X,i}\},\{\phi_{Z,m}\}$ be bases of $H^*(X),H^*(Z)$ respectively. We introduce the following elements of $H_T^*(W)$: \[c_i:=\hat\varphi^*\pr_1^*\phi_{X,i},\quad c_{l,m}:=\hat\jmath_*(\hat{p}^l\hat{\pi}^*\phi_{Z,m}), \textrm{ for } 0\le l\le r-2.\] Their images under the Kirwan map $\kappa_{\td X}(c_i)=\varphi^*\phi_{X,i},\ \kappa_{\td X}(c_{l,m})=\jmath_*(p^l\pi^*\phi_{Z,m})$ form a basis of $H^*(\td X)$. Thus $\{c_i,c_{l,m}\}$ is a $\C[z]\fps{C_{\td X},\bth}$-basis of $\qdm_T(W)_{\td X}\sphat$ by Theorem \ref{thm:qdmiso}. 
Our aim is to show that $\{\Phi(c_i),\Phi(c_{l,m})\}$ is a $\C[z]\laur{\q^{-1/s}}\fps{\sq,\bth}$-basis. It suffices to consider their restrictions to $\sq=\bth=0$. By \eqref{eq:dftlot} and the fact that $\tau(\theta)|_{\sq=\bth=0}=O(\q^{-1})$, $M_X(\tau(\theta))|_{\sq=\bth=0}=\id+O(\q^{-1})$, we have \[\ft_X\sphat(c)|_{\sq=\bth=0}=\kappa_X(c)+O(\q^{-1})\] for $c\in H_T^*(W).$ By \eqref{eq:cftlot}, \[\q^{\rho_Z/((r-1)z)}\cft_{Z,j}(M_W(\theta)c)|_{\sq=\bth=0}=q_{Z,j}e^{h_{Z,j}/z}\lambda_j^n(b+O(\q^{-\frac{1}{r-1}}))\] where $n\in\Z$ and $b\in H^*(Z)$ are given by $i_Z^*c=\lambda^nb+O(\lambda^{n-1}).$ By \eqref{eq:zdiagm} and the fact that $\sigma_j(\theta)|_{\sq=\bth=0}=h_{Z,j}+O(\q^{-\frac{1}{r-1}})$, $M_Z(\sigma_j(\theta))|_{\sq=\bth=0}=e^{h_{Z,j}/z}(\id+O(\q^{-\frac{1}{r-1}}))$, we have \begin{equation}\label{eq:ftlot}
		\ft_{Z,j}\sphat(c)|_{\sq=\bth=0}=q_{Z,j}\lambda_j^n(b+O(\q^{-\frac{1}{r-1}})).
	\end{equation} Since $\kappa_X(c_i)=\phi_{X,i},\ i_Z^*(c_i)=\phi_{X,i}|_Z,\ \kappa_X(c_{l,m})=0,\ i_Z^*c_{l,m}=(-1)^l\lambda^{l+1}\phi_{Z,m}$, we deduce that \[\begin{split}
		\Phi(c_i)|_{\sq=\bth=0}&=\left(\phi_{X,i}+O(\q^{-1}),\left\{q_{Z,j}(\phi_{X,i}|_Z+O(\q^{-\frac{1}{r-1}}))\right\}_{0\le j\le r-2}\right),\\
		(-1)^le^{\frac{\sqrt{-1}\pi r}{r-1}(l+1)}\q^{-\frac{l+1}{r-1}}\Phi(c_{l,m})|_{\sq=\bth=0}&=\left(O(\q^{-\frac{r+l}{r-1}}),\left\{q_{Z,j}(\zeta^{-j(l+1)}\phi_{Z,m}+O(\q^{-\frac{1}{r-1}}))\right\}_{0\le j\le r-2}\right)
	\end{split}\] where $\zeta=e^\frac{2\sqrt{-1}\pi}{r-1}$. Since $(\zeta^{-j(l+1)})_{0\le j,l\le r-2}$ is an invertible matrix, this forms a basis over $\C[z]\laur{\q^{-1/s}}\fps{\sq,\bth}$, and hence $\Phi$ is an isomorphism.
	
	We omit the proof that $\Psi$ respects the pairing, and refer the reader to \cite[Section 5.6.2]{Iri25}.
\end{proof}
\begin{lemma}[{\cite[Lemma 5.11]{Iri25}}]
	For $c\in H_T^*(W)$, we have \begin{equation}\label{eq:dftlot}\begin{split}
		\dft_X(M_W(\theta)c)|_{\sq=0}&=\kappa_X(e^{\theta/z}c)+\sum_{k>0}\q^{-k}\iota_*\left(\frac{\prod_{\nu=1}^{k-1}e_{-\nu z}(\nm_{Z/X})}{k!z^k}[i_Z^*(e^{\theta/z}c)]_{\lambda=kz}\right),\\
		\dft_{\td X}(M_W(\theta)c)|_{\sq=0}&=\kappa_{\td X}(e^{\theta/z}c)+\sum_{k>0}\q^{k}\jmath_*\left(\frac{\prod_{\nu=1}^{k-1}([D]-\nu z)}{\prod_{\nu=1}^{k}e_{-[D]+\nu z}(\nm_{Z/X})}[i_Z^*(e^{\theta/z}c)]_{\lambda=[D]-kz}\right).
	\end{split}
	\end{equation}
\end{lemma}
The lemma is obtained using Proposition \ref{prop:dftdef}, \eqref{eq:shft}, and Lemma \ref{lem:abloc}.
%

To obtain the desired isomorphism of Theorem \ref{thm:blowup}, we restrict the parameter $\theta\in H_T^*(W)$ to the finite dimensional slice $\fds\subset H_T^*(W)$ spanned by the classes $c_i,c_{l,m}$. Let $\{\vartheta^\alpha\}$ be the coordinates dual to $\{c_\alpha\}=\{c_i,c_{l,m}\}.$ Write $\vartheta=\sum_{\alpha}\vartheta^\alpha c_\alpha.$

Using \eqref{eq:dftlot}, \eqref{eq:ftlot}, we obtain:

\begin{lemma}[{\cite[Lemma 5.15]{Iri25}}]\label{lem:jacinv}
	We have \[\begin{split}
		\td\tau(\vartheta)|_{\sq=0}&\in\kappa_{\td X}(\vartheta)+(\vartheta)^2H^*(\td X)[\q]\fps\vartheta,\\
		\tau(\vartheta)|_{\sq=0}&\in\kappa_X(\vartheta)+\q^{-1}[Z]+(\q^{-1}\vartheta,\q^{-2})H^*(X)[\q^{-1}]\fps\vartheta,\\
		\partial_{\vartheta^\alpha}\varsigma_j(\vartheta)|_{\sq=0}&\in\begin{cases}
			\iota^*\phi_{X,i}+\q^{-\frac{1}{r-1}}H^*(Z)[\q^{-\frac{1}{r-1}}] & \textrm{ if } c_\alpha=c_i,\\
			(-1)^l\lambda_j^{l+1}\left(\phi_{Z,m}+\q^{-\frac{1}{r-1}}H^*(Z)[\q^{-\frac{1}{r-1}}]\right) & \textrm{ if } c_\alpha=c_{l,m},
		\end{cases}
	\end{split}\] where $(\vartheta)\subset\C[\q]\fps{\vartheta}$ is the ideal generated by $\{\vartheta^\alpha\}$, and $(\q^{-1}\vartheta,\q^{-2})\subset\C[\q^{-1}]\fps{\vartheta}$ is the ideal generated by $\{\q^{-1}\vartheta^\alpha,\q^{-2}\}$.
\end{lemma}

The change of variables $\fds\to H^*(\td X),\ \vartheta\mapsto\td\tau(\vartheta)$ is invertible over $\C\laur{\q^{-1}}\fps{\sq}$, since its Jacobian matrix at $\sq=\vartheta=0$ is invertible by Lemma \ref{lem:jacinv}. The change of variables $\fds\to H^*(X)\oplus\bigoplus_{j=0}^{r-2} H^*(Z),\ \vartheta\mapsto(\tau(\vartheta),\varsigma_0(\vartheta),\dots,\varsigma_{r-2}(\vartheta))$ is invertible over $\C\laur{\q^{-\frac{1}{r-1}}}\fps{\sq}$ by the same reason. Write $\td\tau\mapsto\vartheta=\vartheta(\td\tau)$ for the inverse map of $\vartheta\mapsto\td\tau(\vartheta).$ Consider the compositions \[\tau(\td\tau):=\tau(\vartheta(\td\tau)),\quad \varsigma_j(\td\tau):=\varsigma_j(\vartheta(\td\tau)).\] Then $\tau(\td\tau)\in H^*(X)\laur{\q^{-1}}\fps{\sq,\td\tau},\ \varsigma_j(\td\tau)\in H^*(Z)\laur{\q^{-\frac{1}{r-1}}}\fps{\sq,\td\tau}$. By Lemma \ref{lem:jacinv}, \[\tau(\td\tau)|_{\sq=\td\tau=0}=\q^{-1}[Z]+O(\q^{-2}).\] Also, we have \[\varsigma_j(\td\tau)|_{\sq=\td\tau=0}=-(r-1)\lambda_j+h_{Z,j}+O(\q^{-\frac{1}{r-1}}).\] Using $\{\kappa_{\td X}(c_\alpha)\}$ as a basis of $H^*(\td X)$ and writing $\td\tau^\alpha$ for the dual coordinates, Lemma \ref{lem:jacinv} implies that \[\begin{split}
	\frac{\partial\tau}{\partial\td\tau^\alpha}(\td\tau)|_{\sq=\td\tau=0}&=\begin{cases}
		\phi_{X,i}+O(\q^{-1}) & \textrm{ if } c_\alpha=c_i,\\
		O(\q^{-1}) & \textrm{ if } c_\alpha=c_{l,m},
	\end{cases}\\
	\frac{\partial\varsigma_j}{\partial\td\tau^\alpha}(\td\tau)|_{\sq=\td\tau=0}&=\begin{cases}
		\iota^*\phi_{X,i}+O(\q^{-\frac{1}{r-1}}) & \textrm{ if } c_\alpha=c_i,\\
		(-1)^l\lambda_j^{l+1}(\phi_{Z,m}+O(\q^{-\frac{1}{r-1}})) & \textrm{ if } c_\alpha=c_{l,m}.
	\end{cases}
\end{split}\]

Restricting the isomorphism $\Psi$ to the finite dimensional slice $\fds$, we obtain an isomorphism over $\C[z]\laur{\q^{-1/s}}\fps{\sq,\td\tau}$ \[\Psi:\qdm(\td X)^\textrm{La}\to\tau^*\qdm(X)^\textrm{La}\oplus\bigoplus_{j=0}^{r-2}\varsigma_j^*\qdm(Z)^\textrm{La}.\] In fact, the maps $\tau(\td\tau)$, $\varsigma_j(\td\tau)$ and the isomorphism $\Psi$ do not involve $x,y$; see \cite[Proposition 5.16]{Iri25} for details. Hence $\tau(\td\tau)\in H^*(X)\laur{\q^{-1}}\fps{Q,\td\tau},\ \varsigma_j(\td\tau)\in H^*(Z)\laur{\q^{-\frac{1}{r-1}}}\fps{Q,\td\tau}$. We introduce \[\qdm(\td X)^\textrm{la}:=H^*(\td X)[z]\laur{\q^{-1/s}}\fps{Q,\td\tau}\] equipped with the pull-back connection \begin{equation*}
\begin{aligned}
\nabla_{\tilde{\tau}^\alpha} & =\partial_{\tilde{\tau}^\alpha}+z^{-1}\phi_{\widetilde{X}, \alpha}{\star_ {\tilde{\tau}}}, \\
\nabla_{z \partial_z} & =z \partial_z-z^{-1}E_{\widetilde{X}}{\star}_{\tilde{\tau}}+\mu_{\widetilde{X}}, \\
\nabla_{\xi Q \partial_Q} & =\xi Q \partial_Q+z^{-1}(\varphi^* \xi) {\star_{\tilde{\tau}}}, \\
\nabla_{\mathfrak{q} \partial_{\mathfrak{q}}} & =\mathfrak{q} \partial_{\mathfrak{q}}-z^{-1}[D]{\star_{\tilde{\tau}}}.
\end{aligned}
\end{equation*} Introduce \[\tau^*\qdm(X)^\textrm{la}:=H^*(X)[z]\laur{\q^{-1/s}}\fps{Q,\td\tau}\] equipped with the pull-back connection \begin{equation*}
\begin{aligned}
\nabla_{\tilde{\tau}^\alpha} & =\partial_{\tilde{\tau}^\alpha}+z^{-1}(\partial_{\tilde{\tau}^\alpha}\tau(\tilde{\tau})){\star_{\tau(\tilde{\tau})}}, \\
\nabla_{z \partial_z} & =z \partial_z-z^{-1}E_{X}{\star}_{\tau(\tilde{\tau})}+\mu_{{X}}, \\
\nabla_{\xi Q \partial_Q} & =\xi Q \partial_Q+z^{-1}\xi {\star_{\tau(\tilde{\tau})}}+z^{-1}(\xi Q \partial_Q\tau(\tilde{\tau})){\star_{\tau(\tilde{\tau})}}, \\
\nabla_{\mathfrak{q} \partial_{\mathfrak{q}}} & =\mathfrak{q} \partial_{\mathfrak{q}}+z^{-1}(\mathfrak{q} \partial_{\mathfrak{q}}\tau(\tilde{\tau})){\star_{\tau(\tilde{\tau})}}.
\end{aligned}
\end{equation*} Introduce \[\varsigma_j^*\qdm(Z)^\textrm{la}:=H^*(Z)[z]\laur{\q^{-1/s}}\fps{Q,\td\tau}\] equipped with the pull-back connection \begin{equation*}
\begin{aligned}
\nabla_{\tilde{\tau}^\alpha} & =\partial_{\tilde{\tau}^\alpha}+z^{-1}(\partial_{\tilde{\tau}^\alpha}\varsigma_j(\tilde{\tau})){\star_{\varsigma_j(\tilde{\tau})}}, \\
\nabla_{z \partial_z} & =z \partial_z-z^{-1}E_{Z}{\star}_{\varsigma_j(\tilde{\tau})}+\mu_{{Z}}, \\
\nabla_{\xi Q \partial_Q} & =\xi Q \partial_Q+z^{-1}(\iota^*\xi) {\star_{\varsigma_j(\tilde{\tau})}}+z^{-1}(\xi Q \partial_Q\varsigma_j(\tilde{\tau})){\star_{\varsigma_j(\tilde{\tau})}}, \\
\nabla_{\mathfrak{q} \partial_{\mathfrak{q}}} & =\mathfrak{q} \partial_{\mathfrak{q}}-z^{-1}(r-1)^{-1}\rho_Z{\star_{\varsigma_j(\tilde{\tau})}}+z^{-1}(\mathfrak{q} \partial_{\mathfrak{q}}\varsigma_j(\tilde{\tau})){\star_{\varsigma_j(\tilde{\tau})}}.
\end{aligned}
\end{equation*} Summarizing the above:
\begin{theorem}[{\cite[Theorem 5.18]{Iri25}}]
	There is an isomorphism \[\Psi:\qdm(\td X)^\textnormal{la}\to\tau^*\qdm(X)^\textnormal{la}\oplus\bigoplus_{j=0}^{r-2}\varsigma_j^*\qdm(Z)^\textnormal{la}\] of $\C[z]\laur{\q^{-1/s}}\fps{Q,\td\tau}$-modules satisfying the properties: \begin{enumerate}[(1)]
		\item $\Psi$ intertwines the quantum connections.
		\item $\Psi$ intertwines the pairings.
		\item The first component of $\Psi$ is homogeneous of degree $0$, and the second component of $\Psi$ is homogeneous of degree $-r$.
		\item $\Psi$ has asymptotics $$
\begin{aligned}
\Psi(\varphi^* \phi_{X, i})|_{Q=\tilde{\tau}=0} & =\left(\phi_{X, i}+O(\mathfrak{q}^{-1}),\left\{q_{Z, j}(\phi_{X, i}|_Z+O(\mathfrak{q}^{-\frac{1}{r-1}}))\right\}_{0 \leq j \leq r-2}\right), \\
\Psi(\jmath_*(p^l \pi^* \phi_{Z, m}))|_{Q=\tilde{\tau}=0} & =\left(O(\mathfrak{q}^{-1}),\left\{q_{Z, j}(-1)^l \lambda_j^{l+1}(\phi_{Z, m}+O(\mathfrak{q}^{-\frac{1}{r-1}}))\right\}_{0 \leq j \leq r-2}\right).
\end{aligned}
$$
	\item The maps $\tau(\td\tau)$ and $\varsigma_j(\td\tau)$ are homogeneous of degree $2$.
	\item $\tau(\td\tau)|_{Q=\td\tau=0}=\q^{-1}[Z]+O(\q^{-2})$ and $\varsigma_j(\td\tau)|_{Q=\td\tau=0}=-(r-1)\lambda_j+h_{Z,j}+O(\q^{-\frac{1}{r-1}}).$
	\item The Jacobian matrix of $(\tau(\td\tau),\varsigma_j(\td\tau))$ at $Q=\td\tau=0$ is given by \[\begin{split}
	\frac{\partial\tau}{\partial\td\tau^\alpha}(\td\tau)|_{Q=\td\tau=0}&=\begin{cases}
		\phi_{X,i}+O(\q^{-1}) & \textrm{ if } c_\alpha=c_i,\\
		O(\q^{-1}) & \textrm{ if } c_\alpha=c_{l,m},
	\end{cases}\\
	\frac{\partial\varsigma_j}{\partial\td\tau^\alpha}(\td\tau)|_{Q=\td\tau=0}&=\begin{cases}
		\iota^*\phi_{X,i}+O(\q^{-\frac{1}{r-1}}) & \textrm{ if } c_\alpha=c_i,\\
		(-1)^l\lambda_j^{l+1}(\phi_{Z,m}+O(\q^{-\frac{1}{r-1}})) & \textrm{ if } c_\alpha=c_{l,m}.
	\end{cases}
\end{split}\]
	\end{enumerate}
\end{theorem}
Theorem \ref{thm:blowup} follows immediately.

\section{Symplectic Irrationality}\label{sec:irr}
In this section we give the proof of Theorem \ref{thm:main}. 

Small quantum cohomology of the quartic threefold $X$ has been computed by Givental \cite{Giv96} and Lian--Liu--Yau \cite{LLY97}. Let $P\in H^2(X)$ be the hyperplane class, so $c_1(X)=P.$ 
\begin{theorem}[{\cite[Corollary 10.9]{Giv96}}]
	The small quantum cohomology algebra of the quartic threefold is isomorphic to \[\C[P,Q]/\left((P+24Q)^4-256Q(P+24Q)^3\right).\] Here $Q$ is the Novikov variable.
\end{theorem}
As an immediate consequence:
\begin{corollary}
	The elements $1,P,P^2,P^3$ (these are powers in the quantum product) form a basis of $H^*(X)[Q]$ over $\C[Q]$. Small quantum multiplication by $P$, viewed as a $\C[Q]$-linear operator on $H^*(X)[Q]$, has characteristic polynomial \[(\lambda+24Q)^3(\lambda-256Q).\] In particular, $-24Q$ is an eigenvalue of (algebraic) multiplicity $3$.
\end{corollary}

Our goal is to show that under bulk deformation by even classes, there still exists an eigenvalue of multiplicty 3:
\begin{theorem}\label{thm:eigen}
	Let $X$ be the quartic threefold. Big quantum multiplication by the Euler vector field $E_X$, viewed as a $\nov{Q,\tau}$-linear operator on $H^*(X)\fps{Q,\tau}$, has an eigenvalue of multiplicity $3$.
\end{theorem}

For the proof, we will need the following lemma.
\begin{lemma}
	Let $E$ be the Euler vector field on an F-manifold. Let $[\cdot,\cdot]$ denote the Lie bracket, and $\circ$ the multiplication. Then \begin{equation}\label{eq:der}
		[E^k,V\circ W]=[E^k,V]\circ W+V\circ[E^k,W]+kE^{k-1}V\circ W
	\end{equation} for $k\ge0$ and vector fields $V,W$.
\end{lemma}
\begin{proof}
	For a vector field $U$, define \[P_U(V,W):=[U,V\circ W]-[U,V]\circ W-V\circ[U,W].\] By \cite[I.2.2.1]{Man99}, the Euler vector field $E$ satisfies the identity \[P_E(V,W)=V\circ W\] for any $V,W$. The F-identity \cite[I.5.1.1]{Man99} implies that \[P_{E^k}=kE^{k-1}P_E.\] Then \[[E^k,V\circ W]-[E^k,V]\circ W-V\circ[E^k,W]=P_{E^k}(V,W)=kE^{k-1}P_E(V,W)=kE^{k-1}V\circ W.\] Rearranging the terms we obtain \eqref{eq:der}.
\end{proof}

We will also need the following commutator identities, which follow from the previous lemma.

\begin{lemma}[{\cite[I.5.6]{Man99}}]
	Let $E$ be the Euler vector field on an F-manifold. Then \begin{equation}\label{eq:comm}
		[E^k,E^l]=(l-k)E^{k+l-1}
	\end{equation} for $k,l\ge0$.
\end{lemma}

We can now give the proof of Theorem \ref{thm:eigen}. Write $E=E_X$.
\begin{proof}[Proof of Theorem \ref{thm:eigen}]
Let $\lambda_1,\dots,\lambda_4$ be the eigenvalues of $E$. The characteristic polynomial of $E$ is \[p(x)=\prod_{i=1}^4(x-\lambda_i).\] By the Cayley-Hamilton theorem, we have \[\prod_{i=1}^4(E-\lambda_i)=0\] where the product is quantum multiplication. Take Lie derivative of the left hand side along the vector field $E^k$, and using the identities \eqref{eq:der}, \eqref{eq:comm}, we obtain \begin{align}
	0&=\sum_{i=1}^4[E^k,E-\lambda_iE^0]\circ\prod_{j\ne i}(E-\lambda_j)+3kE^{k-1}\prod_{i=1}^4(E-\lambda_i) \nonumber\\
	&=\sum_{i=1}^4\left((1-k)E^k-E^k(\lambda_i)+k\lambda_i E^{k-1}\right)\circ\prod_{j\ne i}(E-\lambda_j)+0 \label{eq:key}
\end{align} where $E^k(\lambda_i)$ denotes the vector field $E^k$ differentiating the function $\lambda_i$. Since \[(E-\lambda_i)\circ\prod_{j\ne i}(E-\lambda_j)=0,\] we have that \[E\circ\prod_{j\ne i}(E-\lambda_j)=\lambda_i\prod_{j\ne i}(E-\lambda_j).\] Substituting this into equation \eqref{eq:key}, we obtain \[\sum_{i=1}^4\left(\lambda_i^k-E^k(\lambda_i)\right)\prod_{j\ne i}(E-\lambda_j)=0.\]
Note that this equation has degree 3 in $E$.
Since $1,P,P^2,P^3$ are linearly independent, and $E=P+O(\tau)$, this implies that $1,E,E^2,E^3$ are linearly independent over $\nov{Q,\tau}$. Hence \[E^k(\lambda_i)=\lambda_i^k.\] This can be viewed as a system of ODEs governing $\lambda_i$. The eigenvalues for small quantum cohomology, namely $-24Q$ and $256Q$, give initial values for this system. Clearly, these initial value problems possess unique solutions. In particular, the solution along any $E^k$ for the initial value $-24Q$ is still an eigenvalue of multiplicity 3. Since $1,E,E^2,E^3$ is a basis, the proof is complete.
\end{proof}

\begin{theorem}\label{thm:rat}
	Let $Y$ be a symplectic $6$-manifold which is symplectically rational. Then $E_Y$ has eigenvalues of multiplicity at most $2$.
\end{theorem}
\begin{proof}
	It is folklore that $E_{\cp^n}$ has simple eigenvalues; e.g. it follows from the same argument as Theorem \ref{thm:eigen}. Thus it suffices to show that if $\td M$ is a symplectic blow-up of $M$ along $S$, then $E_{\td M}$ has an eigenvalue of multiplicity at least 3 if and only if $E_M$ has an eigenvalue of multiplicity at least 3. By Theorem \ref{thm:blowup}, the operator $E_{\td M}$ can be identified with $(E_M,E_S,\dots,E_S)$ after an invertible change of variables. Hence $E_{\td M}$ has an eigenvalue of multiplicity at least 3 if and only if $E_M$ or $E_S$ has an eigenvalue of multiplicity at least 3. Now $S$ is either a point or $S^2$, in which case $E_S$ has simple eigenvalues; or $S$ is a surface of genus $>0$, in which case $E_S$ is given by the classical cup product, and has exactly one eigenvalue of multiplicity 2. In any case, $E_S$ does not have an eigenvalue of multiplicity 3. This concludes the proof.
\end{proof}

Theorem \ref{thm:main} follows from Theorem \ref{thm:eigen} and Theorem \ref{thm:rat}.

\appendix

\section{Symplectic virtual techniques}\label{sec:app}
In this appendix, we discuss symplectic virtual techniques.
We construct twisted Gromov--Witten invariants, and prove the quantum Riemann--Roch theorem (Theorem \ref{thm:qrr}) in the symplectic setting following \cite{Coa03}.

\subsection{Global Kuranishi charts and virtual fundamental classes}

We recall the construction of global Kuranishi charts and its properties from \cite{AMS24}. The most pleasant feature of this construction, for our purposes, is that the thickened moduli space admits a smooth submersion to a smooth quasi-projective variety, which is the locus of trivial isotropy of the moduli space of stable maps to $\cp^e$.

Let $\fms_{0,n,e}$ be the subspace of genus 0 stable maps to $\cp^e$ with $n$ marked points of degree $e$, with the property that its image is not contained in any linear subspace of $\cp^e$. Such maps are automorphism free, and $\fms_{0,n,e}$ is a smooth quasi-projective variety, with a universal family $\pi:\uvf_{0,n,e}\to\fms_{0,n,e}$. Henceforth we shall omit the subscripts. Let $\uvf^\circ$ denote the complement of the nodes in $\uvf$. Let $Y:=\Omega^{0,1}_{\uvf^\circ/\fms}\otimes TX$ be the vector bundle over $\uvf^\circ\times X$ whose fiber over $(c,x)$ is the space of anti-holomorphic maps from $T_c(\uvf^\circ|_{\pi(c)})$ to $T_xX$. The group $G=\U(e+1)$ acts on $\cp^e$, and hence on $\fms$ and $Y$. By \cite[Lemma 4.2]{AMS24} we can find a sequence of finite dimensional $G$-representations $V_i$ and $G$-equivariant linear maps $\lambda_i:V_i\to C_c^\infty(Y)$ such that \begin{itemize}
	\item $V_i$ is a subrepresentation of $V_{i+1}$ for each $i$,
	\item $\lambda_{i+1}|_{V_i}=\lambda_i$ for each $i$,
	\item $\bigcup_i\lambda_i(V_i)$ is dense in $C^\infty(Y)$ with respect to the $C^\infty_\textrm{loc}$-topology.
\end{itemize} Such $(V_i,\lambda_i)$ is called a finite dimensional approximation scheme for $C_c^\infty(Y)$. We now fix $i$, and define the pre-thickening $\thck^\textrm{pre}$ to be the set of triples $(\phi,v,u)$, where $\phi\in\fms$, $v\in V_i$, and $u:\uvf|_\phi\to X$ is a smooth map of class $d\in H_2(X,\Z)$ with the properties that \begin{itemize}
	\item $u$ is $J$-holomorphic near its nodes,
	\item $\int_Cu^*\omega\ge0$ on each irreducible component $C$,
	\item $\int_Cu^*\omega\ge\hbar$ on each unstable component $C$, where $\hbar$ is the minimal energy of non-constant $J$-holomorphic spheres;
\end{itemize} the triple $(\phi,v,u)$ is required to satisfy the equation \[\overline\partial_Ju|_{(\uvf^\circ|_\phi)}+\lambda_i(v)\circ\Gamma_u=0\] where $\Gamma_u:\uvf^\circ|_\phi\to\uvf^\circ\times X,\ \sigma\mapsto(\sigma,u(\sigma))$ is the graph. We equip $\thck^\textrm{pre}$ with the topology induced by the Hausdorff distance topology for the closure of the image of $\Gamma_u$ in $\uvf\times X$. The group $G$ acts on $\thck^\textrm{pre}$ in the obvious way. 

We next need to remove the ambiguity coming from incorporating maps to $\cp^e$ in the definition of $\thck^\textrm{pre}$. We choose a Hermitian line bundle $L$ on $X$ whose curvature form is $-2\sqrt{-1}\pi\Omega$, where $\Omega$ is a symplectic form taming $J$. We assume that $e=\deg(u^*L)=[\Omega]\cdot d$, and that it is sufficiently large. We consider ordered $\C$-bases $F=(F_0,\dots,F_e)$ of $H^0(\uvf|_\phi,u^*L)$ such that the Hermitian matrix $H_F:=\left(\int\langle F_i,F_j\rangle u^*\Omega\right)_{i,j}$ has all positive eigenvalues. Such $F$ is called a framing. Then we define the \emph{thickening} $\thck$ to be the space of quadruples $(\phi,v,u,F)$, where $(\phi,v,u)\in\thck^\textrm{pre}$ and $F$ is a framing of $u$. Define the \emph{obstruction bundle} $\obs$ to be the trivial bundle over $\thck$ with fiber $V_i\oplus\mathcal{H}_{e+1}$, where $\mathcal{H}_{e+1}$ is the space of $(e+1)\times(e+1)$ Hermitian matrices. Define the section $s:\thck\to\obs$ by $(\phi,v,u,F)\mapsto(v,\exp^{-1}H_F)$, where $\exp:\mathcal H\to\mathcal H^+$ is the exponential map which sends Hermitian matrices to Hermitian metrices with positive eigenvalues. The group $G$ acts on $\thck$ by the action on $\thck^\textrm{pre}$ and by changing the basis $F$; define the $G$-action on $\obs$ in such a way that $s$ is $G$-equivariant.

\begin{theorem}[{\cite{AMS24}}]\label{thm:ams}
	Assuming that $i$ is sufficiently large, and after shrinking $\thck$ to a neighborhood of $s^{-1}(0)$ (which we still denote by $\thck$), the space $\thck$ is a topological manifold on which $G$ acts with finite stabilizers and finitely many orbit types, and there is a natural homeomorphism \[s^{-1}(0)/G=X_{0,n,d}\] which is an isomorphism of orbispaces. 
\end{theorem}
By gluing theory, the forgetful map $q:\thck\to\fms$ is fiberwise smooth. It is shown in \cite{AMS21} by appealing to abstract smoothing theory that there exists some $G$-representation $V$ such that $\thck\times V$ is a smooth $G$-manifold; we can also ensure that $q\circ p:\thck\times V\to\fms$ and $\ev_i\circ p:\thck\times V\to X$ are smooth submersions, where $p:\thck\times V\to\thck$ is the projection. If we also replace $\obs$ by $p^*\obs\oplus p^*(\thck\times V)$ and $s$ by $p^*s\oplus\Delta$, where $\Delta$ is the tautological section of $p^*(\thck\times V)$, then the $G$-quotient of its zero locus is isomorphic to $X_{0,n,d}$ as well.

The operation of turning $(\thck,\obs,s,G)$ into $(\thck\times V,p^*\obs\oplus p^*(\thck\times V),p^*s\oplus\Delta,G)$ is called stabilization. In fact, we say that two global Kuranishi charts are equivalent if they can be related by the following operations: \begin{enumerate}[(1)]
	\item Germ equivalence: take a $G$-invariant neighborhood $U$ of $s^{-1}(0)$, and replace $(\thck,\obs,s,G)$ by $(U,\obs|_U,s|_U,G)$. This has already been used in the statement of Theorem \ref{thm:ams}.
	\item Stabilization: as defined above.
	\item Group enlargement: replace $(\thck,\obs,s,G)$ by $(P,q^*\obs,q^*s,G\times G')$, where $G'$ is a compact Lie group and $q:P\to\thck$ is a $G$-equivariant principal $G'$-bundle.
\end{enumerate} By \cite[Proposition 4.53, 4.66]{AMS24}, the global Kuranishi chart $(\thck,\obs,s,G)$ is independent of the various choices made during the construction up to equivalence. Thus we may assume that $\thck$ is a smooth $G$-manifold and $\obs$ is a smooth $G$-vector bundle; the forgetful map $\thck\to\fms$ and  evaluation maps $\ev_i:\thck\to X$ are smooth submersions. (The section $s$ is not necessarily smooth.)

To the global Kuranishi chart $(\thck,\obs,s,G)$ we associate a $\Q$-virtual fundamental class as follows. Let $\ofd=\thck/G,\ \zlc=s^{-1}(0),\ \mdl=\zlc/G,\ r=\rank\obs,\ m=\dim\thck-r-\dim G$. For a pair of spaces $(A,B)$, we denote $H_*(A|B;\Q):=H_*(A,A-B;\Q)$. The virtual fundamental class $\vfc{\mdl}$ is defined to be the composition \[\check{H}^m(\mdl;\Q)\xrightarrow[\cong]{\textrm{A-L duality}}{H}_r(\ofd|\mdl;\Q)\cong H_r^G(\thck|\zlc;\Q)\xrightarrow{s_*} 
	H_r^G(\obs|\thck;\Q)\xrightarrow[\cong]{\textrm{Thom}}H_0^G(\thck;\Q)\cong\Q\] where $\check{H}^m(\cdot)$ is \v{C}ech cohomology, ${H}_r(\cdot)$ is Borel-Moore homology, and the first map is Alexander-Lefschetz duality for $\Q$-homology manifolds (see \cite{Skl71}). Thus $[\mdl]^\textrm{vir}$ lives in the Steenrod homology of $\mdl$, which we simply write as $H_m(\mdl;\Q)$ by abuse of notation. It only depends on the global Kuranishi chart up to equivalence. We may also consider $[\mdl]^\textrm{vir}$ with $\C$-coefficients if we wish.
	
For the moduli space $X_{0,n,d}$, we use this procedure to associate the virtual fundamental class $\vfc{X_{0,n,d}}$. Hence we may define the primary Gromov--Witten invariants of $X$. We now explain how to define descendant invariants. Consider the forgetful map $q:\ofd_{0,n,d}\to\fms_{0,n,d}$. We define the $i$th cotangent line bundle $\lb_i$ on $\ofd_{0,n,d}$ to be the pull-back of the $i$th cotangent line bundle on $\fms_{0,n,d}$. Since $\ofd_{0,n,d}$ is stably almost complex, the map $q$ is complex oriented, and hence $\lb_i$ is a complex line bundle. We also denote by $\lb_i$ its restriction to $X_{0,n,d}$. (We shall often use the same notation for a vector bundle on $\ofd_{0,n,d}$ and its restriction to $X_{0,n,d}$ if there is no fear of confusion.) We define the $i$th psi class $\psi_i$ to be the first Chern class of $\lb_i$. Its image in $\check{H}^2(X_{0,n,d};\Q)$ can be capped with the virtual fundamental class. Hence we may define the descendant Gromov--Witten invariants of $X$. Clearly they only depend on the global Kuranishi charts up to equivalence.

We state the fundamental axioms satisfied by the virtual fundamental classes $\vfc{X_{0,n,d}}$. To state the forgetful axiom and splitting axiom, we need some preliminaries. Let $f:X_{0,n+1,d}\to X_{0,n,d}$ be the map forgetting the last marked point. It is in fact the restriction of a forgetful map $f:\ofd_{0,n+1,d}\to\ofd_{0,n,d}$, which is well-defined as long as we make the same choices when constructing $\ofd_{0,n+1,d}$ and $\ofd_{0,n,d}$. Then we have a push-forward in homology $f_*:H_*(\ofd_{0,n+1,d}|X_{0,n+1,d};\Q)\to H_*(\ofd_{0,n,d}|X_{0,n,d};\Q)$, which by Alexander-Lefschetz duality induces a map in \v{C}ech cohomology $f_*:\check{H}^*(X_{0,n+1,d};\Q)\to\check{H}^*(X_{0,n,d};\Q)$, and thus we can make sense of the pull-back $f^*\vfc{X_{0,n,d}}$ as a homology class in $X_{0,n+1,d}$. Let $g:X_{0,n_1+1,d_1}\times_X X_{0,n_2+1,d_2}\to X_{0,n,d}$ be the gluing map. Since the evaluation map is a smooth submersion, the fiber product $\ofd_{0,n_1+1,d_1}\times_X \ofd_{0,n_2+1,d_2}$ is still a smooth orbifold, and thus we can make sense of the pull-back $g^*\vfc{\ofd_{0,n,d}}$ in the same way as before. Let $\Delta$ be the map in the pull-back diagram 
\[\begin{tikzcd}
	{X_{0,n'+1,d'}\times_XX_{0,n''+1,d''}} & {X_{0,n'+1,d'}\times X_{0,n''+1,d''}} \\
	X & {X\times X}
	\arrow["\Delta", from=1-1, to=1-2]
	\arrow["{\ev_{n'+1}=\ev_{n''+1}}"', from=1-1, to=2-1]
	\arrow["{(\ev_{n'+1},\ev_{n''+1})}", from=1-2, to=2-2]
	\arrow["{\textrm{diagonal}}"', from=2-1, to=2-2]
\end{tikzcd}.\] As before, we can make sense of $\Delta^*\left(\vfc{X_{0,n_1+1,d_1}}\times\vfc{X_{0,n_2+1,d_2}}\right)$.
\begin{theorem}[{\cite{AMS24}}]
	The virtual fundamental classes $[X_{0,n,d}]^\textnormal{vir}$ satisfy the following axioms: \begin{enumerate}
		\item[(C)] $\vfc{X_{0,3,0}}=[X]$ under the canonical identification $X_{0,3,0}=X.$
		\item[(F)] $f^*\vfc{X_{0,n,d}}=\vfc{X_{0,n+1,d}}$.
		\item[(P)] $\vfc{X_{0,n,d}}$ is invariant under permutation of marked points.
		\item[(S)] $g^*\vfc{X_{0,n,d}}=\sum\limits_{\substack{n'+n''=n\\d'+d''=d}}\Delta^*\left(\vfc{X_{0,n'+1,d'}}\times\vfc{X_{0,n''+1,d''}}\right)$.
	\end{enumerate}
\end{theorem}
\begin{proof}
	Axiom (C) follows from \cite[Proposition 5.5]{AMS24}. Axiom (F) follows from \cite[Lemma 4.67]{AMS24}. Axiom (P) is \cite[Proposition 5.3]{AMS24}. Axiom (S) follows from \cite[Proposition 4.68]{AMS24}.
\end{proof}

With these axioms in hand, standard arguments lead to the string equation, dilaton equation, divisor equation, topological recursion relations. 
\begin{theorem}\label{thm:eqns}
	For $(n,d)\ne(0,0),(1,0),(2,0)$, we have the string equation \[\langle\alpha_1\psi^{k_1},\dots,\alpha_n\psi^{k_n},1\rangle_{0,n+1,d}=\sum_{i=1}^n\langle\alpha_1\psi^{k_1},\dots,\alpha_{i-1}\psi^{k_{i-1}},\alpha_i\psi^{k_i-1},\alpha_{i+1}\psi^{k_{i+1}},\dots,\alpha_n\psi^{k_n}\rangle_{0,n,d},\] the dilaton equation \[\langle\alpha_1\psi^{k_1},\dots,\alpha_n\psi^{k_n},\psi\rangle_{0,n+1,d}=(n-2)\langle\alpha_1\psi^{k_1},\dots,\alpha_n\psi^{k_n}\rangle_{0,n,d},\] the divisor equation (for $\alpha_{n+1}\in H^2(X;\Q)$) \begin{multline*}
		\langle\alpha_1\psi^{k_1},\dots,\alpha_n\psi^{k_n},\alpha_{n+1}\rangle_{0,n+1,d}=(d\cdot\alpha_{n+1})\langle\alpha_1\psi^{k_1},\dots,\alpha_n\psi^{k_n}\rangle_{0,n,d}\\
		+\sum_{i=1}^n\langle\alpha_1\psi^{k_1},\dots,\alpha_{i-1}\psi^{k_{i-1}},(\alpha_i\cup\alpha_{n+1})\psi^{k_i-1},\alpha_{i+1}\psi^{k_{i+1}},\dots,\alpha_n\psi^{k_n}\rangle_{0,n,d}.
	\end{multline*} Let $\{\phi_i\}$ be a basis of $H^*(X)$, and $(g^{ij})$ the inverse matrix of $\left(\int_X\phi_i\cup\phi_j\right)$. We have the topological recursion relations \begin{multline*}
		\langle\alpha_1\psi^{k_1},\alpha_2\psi^{k_2},\alpha_3\psi^{k_3},\dots,\alpha_n\psi^{k_n}\rangle_{0,n,d}=\\
		\sum_{\substack{d'+d''=d\\ I'\sqcup I''=\{4,\dots,n\}}}\sum_{i,j}g^{ij}\langle\alpha_1\psi^{k_1-1},\alpha_{I'}\psi^{k_{I'}},\phi_i\rangle_{0,|I'|+2,d'}\langle\alpha_2\psi^{k_2},\alpha_3\psi^{k_3},\alpha_{I''}\psi^{k_{I''}},\phi_j\rangle_{0,|I''|+3,d''}
	\end{multline*} where both $I'$ and $I''$ are in order (i.e. we are summing over unshuffles of $\{4,\dots,n\}$), and $\alpha_I\psi^{k_I}$ is multi-index notation representing $|I|$ many insertions.
\end{theorem}
\begin{proof}
In the algebraic setting, these equations follow from the same axioms (F) and (S), combined with various identities relating the psi classes $\psi_i$, forgetful map $f$, and sections $\sigma_i$ (see e.g. \cite{CK99}). Since those identities hold for the algebraic universal family $\fms_{0,n+1,d}\to\fms_{0,n,d}$, they also hold for $\ofd_{0,n+1,d}\to\ofd_{0,n,d}$ by pull-back, and thus for $X_{0,n+1,d}\to X_{0,n,d}$.
\end{proof}

\subsection{Twisted invariants and quantum Riemann--Roch}
Proof of the quantum Riemman-Roch theorem in the symplectic setting is outlined in \cite[Appendix B]{Coa03}, based on work of Siebert \cite{Sie98,Sie99}. We carry out Coates's argument in the genus 0 case using global Kuranishi charts, which are quite similar to Siebert's construction. The main topological input is Atiyah-Hirzebruch's Grothendieck--Riemann--Roch formula in topological K-theory (see e.g. \cite[Corollary 4.18]{Kar78}).

Let $V\to X$ be a Hermitian vector bundle. Write $\ofd_{0,n,d}=\thck_{0,n,d}/G$. Consider the diagram 
\[\begin{tikzcd}
	{\ofd_{0,n+1,d}} & X \\
	{\ofd_{0,n,d}}
	\arrow["\ev", from=1-1, to=1-2]
	\arrow["f"', from=1-1, to=2-1]
\end{tikzcd}\] where by abuse of notation, $\ev$ is the evaluation map at the last marked point, and $f$ is the forgetful map at the last marked point.

We need to make sense of the K-theoretic push-forward $V_{0,n,d}=f_*\ev^*V$. Let $BG^{(i)}$ be the Grassmanian of subspaces in $\C^i$, and $EG^{(i)}$ the space of orthonormal frames in $\C^i$, so $BG=\varinjlim_iBG^{(i)},\ EG=\varinjlim_iEG^{(i)}$. We set $\ofd_{0,n,d}^{(i)}:=(\thck_{0,n,d}\times EG^{(i)})/G$. This is in fact a smooth manifold, and is stably almost complex. Then we define \[K(\ofd_{0,n,d}):=\varprojlim_iK(\ofd_{0,n,d}^{(i)}).\] The forgetful map $f$ induces maps $f^{(i)}$ fitting in the diagram 
\[\begin{tikzcd}
	\cdots & {\ofd_{0,n+1,d}^{(i-1)}} & {\ofd_{0,n+1,d}^{(i)}} & \cdots \\
	\cdots & {\ofd_{0,n,d}^{(i-1)}} & {\ofd_{0,n,d}^{(i)}} & \cdots
	\arrow[from=1-1, to=1-2]
	\arrow[from=1-2, to=1-3]
	\arrow["{f^{(i-1)}}"', from=1-2, to=2-2]
	\arrow[from=1-3, to=1-4]
	\arrow["{f^{(i)}}"', from=1-3, to=2-3]
	\arrow[from=2-1, to=2-2]
	\arrow[from=2-2, to=2-3]
	\arrow[from=2-3, to=2-4]
\end{tikzcd}.\] Each $f^{(i)}$ is proper and complex oriented. Hence we may define the push-forward \[f^{(i)}_*:K(\ofd_{0,n+1,d}^{(i)})\to K(\ofd_{0,n,d}^{(i)})\] by the discussions in \cite[Section 4.5]{Kar78} --- we homotope $f^{(i)}$ to a smooth proper complex oriented map, factor it into an embedding and a projection, and define $f^{(i)}_*$ to be the composition of the Thom isomorphism, collapse map, and Bott periodicity isomorphism. By taking the limit of $f^{(i)}_*$, we obtain \[f_*:K(\ofd_{0,n+1,d})\to K(\ofd_{0,n,d}).\] We define $V_{0,n,d}:=f_*\ev^*V\in K(\ofd_{0,n,d})$, and consider its restriction to $X_{0,n,d}$. Then we may define the twisted Gromov--Witten invariants of $X$. We may also incorporate descendants in the same way as before. Clearly the invariants only depend on the global Kuranishi charts up to equivalence.

Let $T_f:=T\ofd_{0,n+1,d}-f^*T\ofd_{0,n,d}\in K(\ofd_{0,n+1,d})$ be the virtual relative tangent bundle of $f$. Let $\ch$ be the Chern character, $\Td$ the Todd class. 
\begin{lemma}\label{lem:ahrr}
	The Grothendieck--Riemann--Roch formula \[\ch(f_*V)=f_*(\ch(V)\cup\Td(T_f))\] holds for $V\in K(\ofd_{0,n+1,d}).$
\end{lemma}
\begin{proof}
	Observe that $H^*(\ofd_{0,n,d},\Q)\cong H_G^*(\thck_{0,n,d},\Q)\cong\varprojlim_{i}H^*(\ofd_{0,n,d}^{(i)},\Q)$. Indeed, the first isomorphism holds over $\Q$ since $G$ acts on $\thck_{0,n,d}$ with finite stabilizers, and the second isomorphism holds since $EG^{(i)}$ is the $i$-skeleton of $EG$. The desired formula follows by applying \cite[Corollary 4.18]{Kar78} to each $f^{(i)}_*$.
\end{proof}

Let $\sigma_i:\ofd_{0,n,d}\to\ofd_{0,n+1,d}$ be the section of $f$ given by the $i$th marked point. This is also proper and complex oriented, so we may define the push-forward $\sigma_{i*}:K(\ofd_{0,n,d})\to K(\ofd_{0,n+1,d})$ in the same way as before. Let $\iota:\nds\to\ofd_{0,n+1,d}$ be the inclusion of the set of nodes in the fibers. Recall that we have a smooth submersion $\ofd_{0,n+1,d}\to\fms_{0,n+1,d}$, where $\fms_{0,n+1,d}$ is the locus of trivial isotropy in the moduli space of stable maps to $\cp^d$. Since the singular locus in $\fms_{0,n+1,d}$ is a smooth submanifold (of complex codimension 2), so is $\nds\subset\ofd_{0,n+1,d}$. Hence we may define the push-forward $\iota_*:K(\nds)\to K(\ofd_{0,n+1,d})$ as well.
Let $\lb_i$ be the $i$th cotangent line bundle on $\ofd_{0,n+1,d}$, pulled back from $\fms_{0,n+1,d}$. Let $\lb_+,\lb_-$ be the line bundles over $\nds$ formed by the cotangent lines at the nodes. 

\begin{lemma}\label{lem:reltan}
	We have \begin{align*}
		&T_f^*=\mathcal{L}_{n+1}-\sum_{i=1}^n\sigma_{i*}\shf_{\ofd_{0,n,d}}-\iota_*\shf_{\nds},\\
		&T_{\sigma_i}^*=-\lb_i,\\
		&T_{\iota}^*=-\lb_+-\lb_-,
	\end{align*} as elements of $K(\ofd_{0,n+1,d}).$ 
\end{lemma}
\begin{proof}
	By \cite[Equation (1.15)]{Coa03} (see also \cite[Proposition 5.3]{Ton14}) the first equation holds for $\fms_{0,n+1,d}\to\fms_{0,n,d}$; that the other two equations hold is trivial. By pulling back we get the equations for $\ofd_{0,n+1,d}^{(i)}\to\ofd_{0,n,d}^{(i)}$, and in the limit the equations hold for $\ofd_{0,n+1,d}\to\ofd_{0,n,d}$.
\end{proof}

Using Lemma \ref{lem:ahrr} and \ref{lem:reltan} in place of their algebraic counterparts, the same arguments as \cite[Proposition 1.6.3]{Coa03} yields:
\begin{proposition}
	We have \[\ch(V_{0,n,d})= f_*\ev^*\ch(V)\cup\left(\fbox{codim-0}+\fbox{codim-1}+\fbox{codim-2}\right)\] where \begin{align*}
		&\fbox{codim-0}=\Td\lb_{n+1},\\
		&\fbox{codim-1}=-\sum_{i=1}^n\sigma_{i*}\frac{\Td\lb_i}{\psi_i},\\
		&\fbox{codim-2}=\iota_*\left(\frac{1}{\psi_++\psi_-}\left(\frac{\Td L_+}{\psi_+}+\frac{\Td L_-}{\psi_-}\right)\right).
	\end{align*} The last two formulas should be understood as expansions in $\psi_i$ and $\psi_\pm$ respectively, and negative powers are understood to be $0$.
\end{proposition}

We also have the following lemmas from \cite[Appendix B]{Coa03}:
\begin{lemma}[{\cite[Lemma B.3.1]{Coa03}}]
	$f^*\ch(V_{0,n,d})=\ch(V_{0,n+1,d})$.
\end{lemma}

\begin{lemma}[{\cite[Lemma B.3.2]{Coa03}}]
	Consider the gluing map \[g:\bigsqcup_{\substack{n_++n_-=n\\d_++d_-=d}}\ofd_{0,n_++1,d_+}\times_X\ofd_{0,n_-+1,d_-}\to\nds\subset\ofd_{0,n+1,d}.\] We have \[g^*\iota^*\ch(V_{0,n+1,d})=\sum_{\substack{n_++n_-=n\\d_++d_-=d}}\left(p_+^*\ch(V_{0,n_++1,d_+})+p_-^*\ch(V_{0,n_-+1,d_-})-\ev_\Delta^*\ch(V)\right)\] where $p_\pm:\ofd_{0,n_++1,d_+}\times_X\ofd_{0,n_-+1,d_-}\to\ofd_{0,n_\pm+1,d_\pm}$ are the projections, and $\ev_\Delta:\ofd_{0,n_++1,d_+}\times_X\ofd_{0,n_-+1,d_-}\to X$ is the evaluation map at the point of gluing.
\end{lemma}
Then the same arguments as in \cite[Theorem 1.6.4]{Coa03} deduces the quantum Riemann--Roch theorem, or more precisely, the first statement in Theorem \ref{thm:qrr}. To be precise, Coates works with the full potential encoding twisted Gromov--Witten invariants of all genera, which is viewed as a quantized version of the genus zero potential. The key computation is that of the partial derivatives of the potential with respect to the variables $s_i$. For our purposes, carrying out the same computation with the genus zero potential yields the genus zero quantum Riemann--Roch theorem; there is no need to go through higher genus invariants.

The second, slightly more general statement in Theorem \ref{thm:qrr} follows since the operators $\Delta_\alpha$ associated to the different eigenbundles mutually commute (see \cite{Ton14}).

\bibliographystyle{amsalpha}
\bibliography{quartic_irrat.bib}

\end{document}